\documentclass[11pt]{article}
\usepackage{mathtools}
\usepackage[T1]{fontenc}
\usepackage{amsfonts}
\usepackage{amsmath}
\usepackage{amssymb}
\usepackage{amsthm}
\usepackage{bbm}
\usepackage{bm}
\usepackage{mathrsfs}
\usepackage{color}
\usepackage[normalem]{ulem}
\usepackage{pdfsync}
\usepackage{enumitem}
\usepackage[colorlinks=true, linkcolor=blue, citecolor=blue, urlcolor=blue]{hyperref}

\usepackage{tikz}

\newcommand{\RR}{\mathbb{R}}
\newcommand{\R}{\RR}
\newcommand{\N}{\mathbb{N}}

\newcommand{\cL}{\mathcal{L}}

\newcommand{\eps}{ \varepsilon}

\newcommand{\ball}{\mathbb{B}}
\usepackage[margin=31 mm]{geometry}
\newcommand{\mykill}[1]{}
\newcommand{\proofreadhere}{}

\mathtoolsset{showonlyrefs=true}

\usepackage[capitalize, noabbrev]{cleveref}
\crefname{equation}{}{} %
\Crefname{equation}{}{} %
\crefname{enumi}{}{} %
\Crefname{enumi}{}{} %

\theoremstyle{plain}
\newtheorem{theorem}{Theorem}[section]
\newtheorem{proposition}[theorem]{Proposition}
\newtheorem{lemma}[theorem]{Lemma}
\newtheorem{corollary}[theorem]{Corollary}
\theoremstyle{definition}

\newtheorem{remark}[theorem]{Remark}

\crefname{assumption}{Assumption}{Assumptions}
\Crefname{assumption}{Assumption}{Assumptions}
\theoremstyle{remark}
\AddToHook{env/theorem/begin}{\crefalias{section}{theorem}}
\AddToHook{env/proposition/begin}{\crefalias{theorem}{proposition}}
\AddToHook{env/lemma/begin}{\crefalias{theorem}{lemma}}
\AddToHook{env/corollary/begin}{\crefalias{theorem}{corollary}}
\AddToHook{env/definition/begin}{\crefalias{theorem}{definition}}
\AddToHook{env/remark/begin}{\crefalias{theorem}{remark}}
\AddToHook{env/example/begin}{\crefalias{theorem}{example}}
\AddToHook{env/assumption/begin}{\crefalias{theorem}{assumption}}
\crefname{theorem}{Theorem}{Theorems}
\crefname{proposition}{Proposition}{Propositions}
\crefname{lemma}{Lemma}{Lemmas}
\crefname{corollary}{Corollary}{Corollaries}
\crefname{definition}{Definition}{Definitions}
\crefname{remark}{Remark}{Remarks}
\crefname{example}{Example}{Examples}
\crefname{assumption}{Assumption}{Assumptions}

\newlist{myenum}{enumerate}{3}
\setlist[myenum,1]{label={\rm (H\arabic*)},
                   ref  ={\rm (H\arabic*)}}
\crefname{myenumi}{property}{properties}

\renewenvironment{thebibliography}[1]{%
\begin{oldthebibliography}{#1}%
\setlength{\baselineskip}{.9em}
\linespread{1}
\small
\setlength{\parskip}{.30ex}%
\setlength{\itemsep}{.29em}%
}%
{%
\end{oldthebibliography}%
}

\begin{document}

\title{Sharp Asymptotics for Regularized Optimal Transport}
\date{\today}
\author{
Carlos Cardoso-Perelló\thanks{Department of Statistics, Columbia University, cc5415@columbia.edu.}
\and
Alberto Gonz{\'a}lez-Sanz\thanks{Department of Statistics, Columbia University, ag4855@columbia.edu.}
\and
Marcel Nutz\thanks{Department of Statistics, Columbia University, mnutz@columbia.edu.  Research supported by NSF Grants DMS-2106056, DMS-2407074.}
}

\maketitle
\vspace{-1.5em}

\begin{abstract}
We study the small-regularization limit for \(L^p\)-regularized optimal transport with $1<p<\infty$ and for entropically regularized optimal transport (EOT). The exact first-order (respectively, second-order) asymptotics are determined explicitly under mild assumptions on the source and target measures. Our work generalizes the existing results for quadratic and entropic regularization, and connects them by a natural interpolation via \(p\in(1,2)\). We derive all these asymptotics in a unified manner by a novel approach that separates the local computation of the optimal profile from the global enforcement of the marginal constraints: convex duality leads to Gaussian profiles for entropy and Barenblatt profiles for \(L^p\)-regularization, while a quantization construction turns these local profiles into couplings.
\end{abstract}

\vspace{1em}

{\small


}
\vspace{0em}
\setcounter{tocdepth}{2}
\tableofcontents
\section{Introduction}
 The optimal transport (OT) cost between two probability measures  $\mu$ and $\nu$ on $\R^d$ is
\begin{equation}
\label{eq:2_ot_definition}    
 {\rm OT}(\mu,\nu) :=\inf_{\pi\in \Pi(\mu,\nu)}\int \|x-y\|^2 d\pi(x,y), 
\end{equation}
where $\Pi(\mu,\nu)$ is the set of all probability measures on $\R^d\times \R^d$ with marginals~$(\mu,\nu)$.
Practical applications of OT often use regularization to reduce the computational burden and improve the statistical sample complexity. The most prominent example is entropic optimal transport (EOT), adding a relative entropy (Kullback--Leibler) penalization between the coupling $\pi$ and the product $\mu\otimes \nu$ of the marginals: 
\begin{equation} \label{eq:EOT_primal}    
{\rm EOT}_\varepsilon(\mu,\nu) := \inf_{\pi\in \Pi(\mu,\nu)} \left\{\int \|x-y\|^2 d\pi(x,y) + \varepsilon {\rm KL}(\pi|\mu\otimes \nu)\right\},
\end{equation}
where  ${\rm KL}(\mu|\nu)=  \int \log (\frac{d\mu}{d\nu}) d\mu$ if $\mu \ll \nu $ and  ${\rm KL}(\mu|\nu)=+\infty$ otherwise. Rooted in the Schr\"odinger problem \cite{Schrodinger.31}, EOT was popularized in modern times by \cite{cuturi2013sinkhorn} as a tool to approximate the OT cost, based on the fact that Sinkhorn's algorithm \cite{Sinkhorn.64} for the dual EOT problem converges linearly (e.g., \cite{FranklinLorenz.89,ChenGeorgiouPavon.16,Carlier.21,GhosalNutz.22,ConfortiDurmusGreco.23,Eckstein.25,ChizatDelalandeVaskevicius.25}). See also \cite{Leonard.14,PeyreCuturi.19,Nutz.20} for background. Another important property for applications, established by \cite{genevay2019sample, MenaWeed.19} and others, is that EOT avoids the statistical curse of dimensionality suffered by OT; see also~\cite{delbarrio.et.al.2025.survey,balakrishnan.2025.statisticalinferenceoptimaltransport,ChewiNilesWeedRigollet.25} for surveys.

While the KL penalty entails these desirable properties, it also forces the optimal EOT coupling to have full support, while unregularized optimal transport is typically sparse, and weak regularization can cause numerical instability in EOT
through exponentially large and small scalings in the dual problem~\cite{Schmitzer.19}. More recently, starting with \cite{Muzellec.2017.AAAI,blondel18quadratic,EssidSolomon.18,LorenzMannsMeyer.21}, a growing literature explores alternative regularizations, especially quadratically regularized optimal transport (QOT, using the $L^2$ penalty) and more generally $L^p$-regularized optimal transport ($p$-ROT) for $1<p<\infty$, defined as
\begin{equation}
    \label{eq:lq_ot}
{\rm ROT}_{\eps,p}(\mu,\nu):=  \inf_{\pi\in \Pi(\mu,\nu)}\left\{\int \|x-y\|^2 d\pi(x,\,y) + \varepsilon\left\|\frac{d\pi}{d(\mu\otimes \nu)}\right\|_{L^p(\mu\otimes \nu)}^p\right\}
\end{equation}
where $\|\frac{d\mu}{d\nu}\|_{L^p(\nu)}=+\infty$ if $\mu \not \ll \nu$.
It produces sparse couplings, as observed empirically in \cite{blondel18quadratic,EssidSolomon.18,BayraktarEckstein.2025.BJ}
and analyzed theoretically in \cite{Nutz.2024,WieselXu.24,GonzalezSanzNutz2024.Scalar,gonzalezsanz2026sharplocalsparsityregularized}. It also replaces the exponential structure in the EOT dual by a polynomial, thereby avoiding exponential scaling. Recent works have shown linear convergence of several dual algorithms~\cite{GonzalezSanzNutzRiveros.25,GonzalezSanzNutzRiveros.26} as well as parametric sample complexity~\cite{GonzalezSanzEcksteinNutz.25,GonzalezSanzDelBarrioNutz.25,HanWiesel.26,GonzalezSanzNutzStromme.26}, further supporting the use of $p$-ROT as an alternative to EOT in applications (e.g., \cite{zhang2023manifold}).

For $p>1$, $L^p$-regularization is equivalent to regularizing by the $p$-Tsallis divergence, up to rescaling $\eps$ to $\eps/(p-1)$ and adding a constant. The $p$-Tsallis divergence converges to ${\rm KL}$ as $p\to1^+$, so that EOT can be seen as a limiting case of $L^p$-regularization after rescaling.

\paragraph{Contributions.}
Our two main contributions are (i)~the exact asymptotics for 
EOT and $p$-ROT ($1<p<\infty$) under very mild conditions and (ii)~a novel quantization technique enabling these derivations. Our results substantially generalize the existing results for EOT and for quadratically regularized optimal transport (QOT, meaning that $p=2$) while also connecting them with the natural interpolation via $p\in(1,2)$. Moreover, our technique gives a unified approach to derive all these asymptotics. %

Specifically, we show under mild assumptions on $\mu$ and $\nu$ that
\begin{align}
\label{eq:eot_limit_intro}
\lim_{\varepsilon\to0^+}\frac{{\rm EOT}_\eps(\mu,\nu) - {\rm OT}(\mu,\nu) + \frac{d}{2} \eps\,\log(\pi\varepsilon)}{\varepsilon} &= -\frac{{\int \log f_\mu\, d\mu+\int\log f_\nu \,d\nu}}{2},
\\ 
\label{eq:prot_limit_intro}
 \lim_{\varepsilon \to 0^+}\frac{{\rm ROT}_{\varepsilon,p}(\mu,\nu)-{\rm OT}(\mu,\nu)}{\varepsilon^{\frac{2}{d(p-1)+2}}}&= \mathfrak C_{d,p}\cdot\mathcal{K}(\mu,\nu,p),\quad \text{for all}\ p>1,
\end{align}
where we write $f_P$ for the density of a measure $P$, $\mathfrak C_{d,p}$ is the dimensional constant~\eqref{eq:def-B-C-constants}, and $\mathcal{K}(\mu,\nu,p)$ is the value of the explicit integral~\eqref{eq:def-K-constant} involving the marginals and the optimal transport map between them. While quite different at first glance, we will see in \cref{sec:entropic-endpoint} that the expansions~\eqref{eq:eot_limit_intro} and~\eqref{eq:prot_limit_intro} are consistent in the same way as $L^p$-regularization converges to relative entropy for $p\to1^+$ after appropriate rescaling.

Beyond identifying these limits, our results substantially enlarge the range in which sharp small-regularization asymptotics are known. For EOT and $p$-ROT with \(1<p\le2\), the marginals are only assumed to be absolutely continuous with a finite \(2+\beta\) moment for some~$\beta>0$ (in addition to standard assumptions on the optimal transport map). In contrast to previous results, no compactness, boundary regularity, or upper or lower bounds on the densities are required. We even cover the situation where the right-hand sides of~\eqref{eq:eot_limit_intro} and~\eqref{eq:prot_limit_intro} are infinite. 

While the proof of the lower bound proceeds by duality, the main obstacle is the upper bound, as it requires an approximately optimal coupling matching the marginal constraints exactly. At the local level, the dual problem selects a canonical profile: Gaussian for EOT and Barenblatt for $p$-ROT. Naturally, these profiles agree with the ones appearing in the earlier EOT and QOT analyses of \cite{pal2019difference,GarrizmolinaElAl.2024} in the corresponding special cases. The main novelty in our approach is how the local profile is made into a genuine coupling.  Rather than forcing the full approximate density to satisfy the marginal constraints by a problem-specific global correction, we use the profile only to build a subcoupling whose marginals are dominated by the prescribed ones, and then complete the remaining marginal mass by a quantization construction in a second step. This separates the computation of the local asymptotic constant from the enforcement of the marginal constraints. Our approach explains the Gaussian-to-Barenblatt transition in a unified way and provides a flexible route to sharp asymptotics for other convex regularizations.

While the general recipe is the same for all regularizations, it turns out that there are two distinct regimes, depending on the convexity properties of the regularizer. The first regime includes EOT and $p$-ROT with \(p\leq2\) (and in particular all previously analyzed cases); here the quantization argument can be carried out under very general conditions. By contrast, $p$-ROT with \(p>2\) falls into the second regime; here, the quantization has to be carried out with cells of comparable size, and that is why additional conditions on the marginals are necessary when the limit is finite. See \cref{sec:general_recipe} for a more detailed sketch of the approach.

\paragraph{Relation to previous results.}
The asymptotics for the special cases of EOT and QOT were previously found under stronger assumptions and with different techniques. The limit~\eqref{eq:eot_limit_intro} for EOT was first shown by \cite{pal2019difference} when the marginals $\mu,\nu$ have compact supports satisfying a uniform interior cone condition, and moreover $\mu,\nu$ admit densities $f_\mu,f_\nu$ which are continuous and uniformly bounded away from zero on the supports. The proof constructs an asymptotically Gaussian perturbation of the optimal transport map and then uses a Markov construction to enforce the terminal marginal. %
In the Schr\"odinger bridge setting, the corresponding short-time asymptotics were shown by \cite{ConfortiTamanini.19}  on a complete Riemannian manifold under a Bakry--\'Emery curvature lower bound, with compactly supported marginals having bounded densities
relative to the reversible measure, and under differentiability of the rescaled entropic cost and of
the integrated Fisher information. The proof proceeds through
the Benamou--Brenier formulation and Fisher-information estimates. 

The approach of the present paper allows us to derive \eqref{eq:eot_limit_intro} for absolutely continuous marginals with finite moments of order $2+\beta$ and without conditions on the support. In particular, this is the first such result allowing for marginals with unbounded support. Moreover, our result holds even when the right-hand side of \eqref{eq:eot_limit_intro} has the value $-\infty$ (cf.\ \cref{rk:finitenessOfLimit}).

For QOT, \cite{GarrizmolinaElAl.2024} showed the special case $p=2$ of \eqref{eq:prot_limit_intro} when the marginals $\mu,\nu$ have compact supports with a uniform interior cone condition and admit densities $f_\mu,f_\nu$ which are H\"older continuous and bounded away from zero. Their
lower bound is based on a Barenblatt-type dual ansatz, while the upper
bound enforces the marginal constraints through a global transport
correction, supplemented by separate ``frame'' and ``stained-glass''
couplings near the boundary. 

The common ingredient in the prior works is the identification of the local
Gaussian or Barenblatt profile, whereas the principal difference lies in the
enforcement of the marginal constraints: the upper-bound proofs in
\cite{pal2019difference,GarrizmolinaElAl.2024} combine the local profile with
problem-specific global corrections---a Markov-kernel composition in the
former case and transport together with boundary patching in the latter.
Our method separates these two roles clearly: the local profile determines the
asymptotic constant, while a quantization-based completion is added to the local construction to fit the prescribed marginals. This quantization approach is a key innovation in the present work, permitting a substantially more general setting and a unified treatment of different regularizations. Arguably, it is also simpler than the previous techniques. Naturally, the local profiles remain Gaussian and Barenblatt-type, respectively, for EOT and QOT, while the Barenblatt profile is further extended to general $p>1$. Our proof approach is sketched in more detail in \cref{sec:general_recipe}.

For \(p\neq2\), we are not aware of previous results identifying the limit \eqref{eq:prot_limit_intro}. Coarser results, however, exist for more general transport costs and regularizations. For infinitesimally twisted transport costs (which include $\|x-y\|^2$) and compactly supported marginals with bounded densities, \cite{Carlier-Pegon-Tamanini} showed upper and lower bounds for the quotient on the left-hand side of~\eqref{eq:eot_limit_intro}; however it remains open whether there are matching bounds, i.e., whether the limit exists and what its value might be. Their proof uses block approximations and an integral Alexandrov estimate for the upper bound, and a quadratic detachment based on Minty's trick for the lower bound. These techniques are likely not sharp enough to obtain the limit in \eqref{eq:eot_limit_intro}. For general \(f\)-divergence regularizations,
\cite{Eckstein.Nutz.Quantization.MOR} obtained nonasymptotic rates through the  ``shadow'' coupling technique and the arguments of \cite{Carlier-Pegon-Tamanini}. For quadratic cost and \(L^p\)-regularization, their bounds yield the exponent \(\frac{2}{d(p-1)+2}\), and for EOT, the term $(d/2) \varepsilon\log(1/\varepsilon)$. Again, the bounds are not sharp enough to obtain the limits in \eqref{eq:eot_limit_intro}--\eqref{eq:prot_limit_intro}. Different rates arise in the context of the distance cost $\|x-y\|$, where the unregularized OT problem has multiple solutions. Here, the second-order asymptotics of EOT were recently found by \cite{NutzZhong.26}.

The aforementioned works are preceded by a large literature establishing the (qualitative) vanishing-regularization limit in various settings, which we do not review here. See, e.g., \cite{Leonard.12,Leonard.14,CarlierDuvalPeyreSchmitzer.17} and the references therein.

\paragraph{Organization of the paper.} \Cref{Section:main-results} states the main results, discusses the relation between the EOT and $p$-ROT limits as $p\to1^+$, and sketches the proof idea.  \Cref{Section:lower} proves the lower bound for $p$-ROT and EOT. \Cref{Section:main-upper} presents the proof of the matching upper bound, which is our main technical contribution. \Cref{se:omitted-proofs} gathers omitted proofs.

\section{Main results}\label{Section:main-results}

\subsection{General notation} \label{sec:Notation}
We denote by \(\cL^d\) the Lebesgue measure on
\(\mathbb R^d\); we also use $|A|=\cL^d(A)$ where convenient.   For \(x\in\mathbb R^d\) and \(R>0\), we write
\(\ball(x,R):=\{y\in\mathbb R^d:\|x-y\|<R\}\). For a set
\(E\subset\mathbb R^d\), interior and boundary are denoted by
\(\operatorname{int}(E)\) and \(\partial E\), respectively, and
\(x+E:=\{x+y:y\in E\}\). For a positive-definite matrix \(A\), we denote
\(\|x\|_A^2:=\langle x,  A x\rangle\), and we write
\(\omega_{d-1}:=\mathcal H^{d-1}(\mathbb S^{d-1})\) for the
\((d-1)\)-dimensional Hausdorff measure of the unit sphere
\(\mathbb S^{d-1}\subset\mathbb R^d\). We use
\(0\log0:=0\).

Fix a Borel set $S\subseteq \mathbb{R}^d$. Let $\mathcal{M}^+(S)$ be the set of finite positive Borel measures supported on $S$; we identify $\mathcal{M}^+(\mathbb{R}^d\times \mathbb{R}^n)$ with $\mathcal{M}^+(\mathbb{R}^{d+n})$. For \(q>0\), \(\mathcal P_q(S)\) denotes the set of Borel
probability measures on $S$ with finite \(q\)-th moment, and
\(\mathcal P_q^{\rm ac}(S)\) denotes the subset of $\mathcal P_q(S)$ consisting of measures $\mu$ which are
absolutely continuous with respect to \(\cL^d\), denoted \(\mu\ll\cL^d\). In that case, we write \(f_\mu\) for the density, so that
\(\mu=f_\mu\,d\cL^d\). For a Borel set \(E\), \(\mu|_E\) denotes the
restriction of \(\mu\) to \(E\).

\subsection{Assumptions}

The first two assumptions on the marginals $\mu,\nu$ are used for all of our main results.

\begin{enumerate}[label={\rm (A\arabic*)}, leftmargin=*, itemsep=.35em]
 
    \item \label{assumption:hoelder_densities}
    The probability measures $\mu$ and $\nu$ satisfy $\mu,\nu \in \mathcal{P}_2^{\rm ac}(\mathbb{R}^d)$. Therefore, there is a $\mu$-a.s.\ unique gradient $\nabla\varphi$ of a convex function~$\varphi:\R^d\to\R\cup\{\infty\}$ pushing~$\mu$ forward to $\nu$.
    \item \label{assumption:elliptic_phi}
    The potential $\varphi$ can be chosen in ${\cal C}^{2}(\R^d)$ and such that
    \begin{equation}
        \label{eq:Elliptic}
                \sigma_m(\varphi)I
        \preceq
        \nabla^2\varphi(y)
        \preceq
        \sigma_M(\varphi)I,
        \qquad y\in\R^d
    \end{equation}
for some constants
    \(0<\sigma_m(\varphi)\leq \sigma_M(\varphi)<\infty\).
\end{enumerate}

Assumption~\ref{assumption:elliptic_phi} holds, for example, when \(\mu=\nu\), with
\(\varphi(x)=\|x\|^2/2\).  It is also the standard conclusion of
the classical regularity theory for the Monge--Amp\`ere equation under appropriate smoothness, convexity, and two-sided density bounds. Since the arguments below merely use the existence of $\sigma_m(\varphi)$ and $\sigma_M(\varphi)$, we state this condition directly.

The following two additional assumptions will be used only for $p$-ROT with $p>2$.
 
\begin{enumerate}[label={\rm (H\arabic*)}, leftmargin=*, itemsep=.35em]
    \item \label{assumption:cone-2}
   The interiors $\Omega_0={\rm int}\,{\rm supp}\,\mu$ and $\Omega_1={\rm int}\,{\rm supp}\,\nu$ are bounded sets with Lipschitz boundaries\footnote{We say that an open set \(Q\subset \mathbb R^d\) has
Lipschitz boundary if there exist constants \(L,r_0>0\) such that
for every \(x_0\in \partial Q\), there exist an orthogonal map \(R:\mathbb R^d\to\mathbb R^d\) and an $L$-Lipschitz function \(g:\mathbb R^{d-1}\to\mathbb R\) with $g(0)=0$ 
for which the image of \(Q\) in the coordinates 
\(R(x-x_0)=(x',x_d)\in\mathbb R^{d-1}\times\mathbb R\) satisfies
\[
R\bigl((Q-x_0)\cap \ball(0,r_0)\bigr)
=
\left\{
(x',x_d)\in \ball(0,r_0):
x_d>g(x')
\right\}.
\]
The constants \(L\) and $r_0$ are called the Lipschitz constant and the localization radius of~$Q$.} satisfying $\operatorname{supp}\mu=\overline{\Omega_0}$ and $
\operatorname{supp}\nu=\overline{\Omega_1}$. 
    \item \label{assumption:hoelder_densities-2}
    The densities  $f_\mu$ and $f_\nu$ are continuous and uniformly bounded and bounded away from zero on ${\rm supp}(\mu)$ and ${\rm supp}(\nu)$.  
\end{enumerate}

\subsection{Results}

We first state the asymptotics for $p$-ROT (where $p>1$). The limit involves the quantity
\begin{align}\label{eq:def-K-constant}
\mathcal{K}(\mu,\nu,p):= \int {[f_\nu(\nabla \varphi(x))f_\mu(x)]^{-\frac{p-1}{d(p-1)+2}}}d\mu(x) \in [0,\infty],
\end{align}
as well as the dimensional (and $p$-dependent) constants 
\begin{align}\label{eq:def-B-C-constants}
\mathfrak B_{d,p}:=\frac{p^{d/2}}{2}\omega_{d-1}B\left(\frac{p}{p-1},\frac d2\right),
\qquad
\mathfrak C_{d,p}:=\frac{p(d(p-1)+2)}{d(p-1)+2p}\,\mathfrak B_{d,p}^{-\frac{2(p-1)}{d(p-1)+2}},
\end{align}
where $\omega_{d-1}$ is the $(d-1)$-dimensional Hausdorff measure of the unit sphere $\mathbb S^{d-1}$ and $B(\cdot,\cdot)$ is the beta function. The (in)finiteness of $\mathcal K(\mu,\nu,p)$ is discussed in \cref{rk:finitenessOfLimit} below.

\begin{theorem}[$p$-ROT]\label{theorem:main-p} Let $1<p<\infty$ and let \ref{assumption:hoelder_densities} and \ref{assumption:elliptic_phi} hold. The limit 
    \begin{equation}
    \label{eq:limit-in-theorem}
    \lim_{\varepsilon \to 0^+}\frac{{\rm ROT}_{\varepsilon,p}(\mu,\nu)-{\rm OT}(\mu,\nu)}{\varepsilon^{\frac{2}{d(p-1)+2}}}= \mathfrak C_{d,p}\cdot\mathcal{K}(\mu,\nu,p)
    \end{equation}
    holds in any of the following cases:
    \begin{enumerate}
    \item $p\in(1,2]$ and $\mu,\nu\in\mathcal P_{2+\beta}(\mathbb R^d)$ for some $\beta>0$, 
    \item  $p>2$ and Assumptions~\ref{assumption:cone-2} and \ref{assumption:hoelder_densities-2} hold,
    \item \(\mathcal K(\mu,\nu,p)=\infty\).
    \end{enumerate}
    In cases (i) and (ii), we have $\mathcal K(\mu,\nu,p)\in\R$ and hence the limit \eqref{eq:limit-in-theorem} is finite.
\end{theorem}

To state the main result for EOT, we define the constant
\begin{align}\label{eq:def-K-constant-EOT}
    \mathcal K_{\rm EOT}(\mu,\nu)& :=-\frac12
\int_{\mathbb R^d}
\log\left[
f_\mu(x)f_\nu(\nabla\varphi(x))
\right]d\mu(x)= -\frac{{\rm Ent}(\mu)+{\rm Ent}(\nu)}{2} \in [-\infty,\infty),
\end{align}
where ${\rm Ent}(\alpha):=\int \log
f_\alpha
\, d\alpha$ for $\alpha \in \mathcal{P}^{\rm ac}_2(\mathbb{R}^d)$. We have ${\rm Ent}(\alpha)\in (-\infty,\infty]$ due to the finite second moment, as can be seen by comparing with a Gaussian. 

\begin{theorem}[EOT]\label{theorem:main-EOT} Let \ref{assumption:hoelder_densities} and \ref{assumption:elliptic_phi} hold, and let $\mu,\nu\in\mathcal P_{2+\beta}(\mathbb R^d)$ for some $\beta>0$. Then 
    \begin{equation}
        \label{eq:EOT_main_result}
\lim_{\varepsilon\to 0^+}
\frac{
{\rm EOT}_\varepsilon(\mu,\nu)
-
{\rm OT}(\mu,\nu)
+
\frac d2\,\varepsilon\log(\pi\varepsilon)
}{\varepsilon}
=
\mathcal K_{\rm EOT}(\mu,\nu).
    \end{equation}
\end{theorem}

We emphasize that~\eqref{eq:EOT_main_result} holds regardless of whether $\mathcal K_{\rm EOT}(\mu,\nu)$ is finite. See also~\cref{rk:finitenessOfLimit}.

\begin{remark}[(In)finiteness of the constants]\label{rk:finitenessOfLimit}

\begin{enumerate}[label={\rm (\roman*)}, leftmargin=*]
    \item $\mathcal K_{\rm EOT}(\mu,\nu)$ is finite if $\mu,\nu$ have bounded densities. Compact support does not suffice: there exist compactly supported $\mu,\nu$ satisfying \cref{assumption:hoelder_densities,assumption:elliptic_phi} such that $\mathcal K_{\rm EOT}(\mu,\nu)=-\infty$.

    \item Let $p\in(1,2]$. If $\mu,\nu\in\mathcal P^{\rm ac}_{2+\beta}(\mathbb R^d)$ for some $\beta>0$, then $\mathcal K(\mu,\nu,p)<\infty$. Positivity of $\beta$ cannot be dropped: there exist $\mu,\nu$ satisfying \cref{assumption:hoelder_densities,assumption:elliptic_phi} such that $\mathcal K(\mu,\nu,2)=\infty$.

    \item Let $p>2$. If $\mu,\nu\in\mathcal P^{\rm ac}(\mathbb R^d)$ satisfy \cref{assumption:cone-2,assumption:hoelder_densities-2}, then $\mathcal K(\mu,\nu,p)<\infty$.\footnote{More generally, lower bounds $f_\mu,f_\nu\geq m>0$ on the respective supports already imply $\mathcal K(\mu,\nu,p)<\infty$, for any $p>1$.} Finite $2+\beta$ moments alone do not imply finiteness of $\mathcal K(\mu,\nu,p)$ for $p>2$: for every $p>2$ there exist $\beta>0$ and $\mu,\nu\in\mathcal P^{\rm ac}_{2+\beta}(\mathbb R^d)$ satisfying \cref{assumption:hoelder_densities,assumption:elliptic_phi} such that $\mathcal K(\mu,\nu,p)=\infty$. Compact support is sufficient if $d\geq2$, and also if $d=1$ and $p\leq3$; but not if $d=1$ and $p>3$, even if the densities are bounded and continuous on their supports.
\end{enumerate}
All counterexamples above may be chosen with $\mu=\nu$, so that the Brenier map $\nabla \varphi$ is the identity and \cref{assumption:elliptic_phi} holds with $\varphi(x)=\|x\|^2/2$.
\end{remark}

The proofs/counterexamples for \cref{rk:finitenessOfLimit} are detailed in \cref{se:omitted-proofs}. The counterexamples in \cref{rk:finitenessOfLimit}(iii), together with \cref{theorem:main-p}(iii), contradict the assertion of \cite[Corollary~3.14]{Eckstein.Nutz.Quantization.MOR} when $p>2$. This is due to a glitch in~\cite{Eckstein.Nutz.Quantization.MOR}: as its derivation uses concavity of $x^{p-1}$, the corollary should assume $1<p\leq 2$ (rather than just $p>1$, as stated in~\cite{Eckstein.Nutz.Quantization.MOR}).

\begin{remark}
The $2+\beta$ moment assumption in \cref{theorem:main-p}(i) and \cref{theorem:main-EOT} is related to the quadratic transport cost rather than the choice of regularizer. The assumption is used in the quantization step in the proof of the upper bound, where the exponent \(2\) is dictated by the transport cost that must be controlled when completing the marginal defect. The slack \(\beta\) is used only to make the tail contribution of the remainder coupling (cf. \cref{sec:general_recipe}) negligible at the relevant scale. 
\end{remark}

\subsection{Relation between the $p$-ROT and EOT limits}
\label{sec:entropic-endpoint}

The \(L^p\)-regularized problem considered in \cref{theorem:main-p} uses the
unnormalized regularizer \(\Upsilon(r)=r^p\), which degenerates for \(p\to1^+\). To connect to EOT, we use the Tsallis divergence
\[
\Upsilon_p^{\rm Ts}(r)
:=
\frac{r^p-pr+p-1}{p-1},
\qquad r\geq0,
\]
which is equivalent to \(L^p\)-regularization up to constants. Specifically, note that  \(\Upsilon_p^{\rm Ts}(1)=0\), \((\Upsilon_p^{\rm Ts})'(1)=0\), and
\(\Upsilon_p^{\rm Ts}(r)\to r\log r-r+1\) as \(p\to1^+\). Denote by \({\rm TROT}_{\varepsilon,p}(\mu,\nu)\) the resulting regularized transport problem, i.e.,
\[
{\rm TROT}_{\varepsilon,p}(\mu,\nu)
:=
\inf_{\pi\in\Pi(\mu,\nu)}
\left\{
\int \|x-y\|^2\,d\pi(x,y)
+
\varepsilon
\int
\Upsilon_p^{\rm Ts}
\left(
\frac{d\pi}{d(\mu\otimes\nu)}
\right)
d(\mu\otimes\nu)
\right\}.
\]
Then 
\[
{\rm TROT}_{\varepsilon,p}(\mu,\nu)
=
{\rm ROT}_{\varepsilon/(p-1),p}(\mu,\nu)
-
\frac{\varepsilon}{p-1}
\]
for any \(p>1\) and \(\varepsilon>0\), so that the limit~\eqref{eq:limit-in-theorem} translates to 
\begin{align}\label{eq:TROT-asymptotic}
\lim_{\varepsilon\downarrow0}
\frac{
{\rm TROT}_{\varepsilon,p}(\mu,\nu)
-
{\rm OT}(\mu,\nu)
+
\frac{\varepsilon}{p-1}
}{
\bigl(\frac{\varepsilon}{p-1}\bigr)^{\frac{2}{d(p-1)+2}}
}
=
\mathfrak C_{d,p}\mathcal K(\mu,\nu,p).
\end{align}
The following shows that the asymptotic~\eqref{eq:TROT-asymptotic} is consistent with the EOT asymptotic~\eqref{eq:EOT_main_result} in the limit $p\to1^+$.

\begin{lemma}[Consistency as \(p\to1^+\)]
\label{le:consistency-p-to-one} Assume \ref{assumption:hoelder_densities} and \ref{assumption:elliptic_phi} hold, that $\mu,\nu\in\mathcal P_{2+\beta}(\mathbb R^d)$ for some $\beta>0$, and 
\begin{equation}
\label{eq:finite_LlogL}
    \int |\log f_\mu|\,d\mu+\int|\log f_\nu|\,d\nu<\infty.
\end{equation}
Then, for every fixed \(\varepsilon>0\),
\[
\lim_{p\to1^+}
\left\{
\left(\frac{\varepsilon}{p-1}\right)^{\frac{2}{d(p-1)+2}}
\mathfrak C_{d,p} \mathcal K(\mu,\nu,p)
-
\frac{\varepsilon}{p-1}
\right\}
=
-\frac d2\,\varepsilon\log(\pi\varepsilon)
+
\varepsilon\mathcal K_{\rm EOT}(\mu,\nu).
\]
\end{lemma}

To be clear, \cref{le:consistency-p-to-one} does not imply that one can rigorously infer the EOT asymptotics~\eqref{eq:EOT_main_result} from~\eqref{eq:TROT-asymptotic} (and hence from \cref{theorem:main-p}). That would require additional uniformity properties in~$p$ to justify a double limit ($p\to1^+$ and $\eps\to0^+$). While such an argument is possible, it requires stronger conditions and a more complicated proof than the direct derivation of \cref{theorem:main-EOT} that we give below.

\subsection{Proof idea}\label{sec:general_recipe}

We sketch the proof idea for a general convex regularizer \(\Upsilon:[0,\infty)\to(-\infty,\infty]\) and start by formally deriving the local profile of the ROT solution for small~$\eps$. Let
\[
\Xi(s):=\sup_{r\geq0}\{rs-\Upsilon(r)\},\qquad s\in\mathbb R
\]
be the convex conjugate and consider the ROT problem
\[
{\rm ROT}_{\varepsilon,\Upsilon}(\mu,\nu)
:=
\inf_{\pi\in\Pi(\mu,\nu)}
\left\{
\int \|x-y\|^2\,d\pi(x,y)
+
\varepsilon
\int
\Upsilon\left(\frac{d\pi}{d(\mu\otimes\nu)}\right)
\,d(\mu\otimes\nu)
\right\}.
\]
Assuming that Fenchel--Rockafellar duality is applicable, the corresponding dual is
\[
\sup_{f,g}
\bigg\{
\int f\,d\mu+
\int g\,d\nu
-
\varepsilon
\int
\Xi\left(
\frac{f(x)+g(y)-\|x-y\|^2}{\varepsilon}
\right)
\,d\mu(x)d\nu(y)
\bigg\},
\]
and the associated optimal density w.r.t.\ \(\mu\otimes\nu\) should satisfy the first-order condition
\begin{equation}
h_\varepsilon(x,y)
=
\Xi'\left(
\frac{f_\varepsilon(x)+g_\varepsilon(y)-\|x-y\|^2}{\varepsilon}
\right)
\end{equation}
where \(\Xi'\) is the derivative (understood in the subgradient sense if $\Xi$ is not differentiable). We now expand this formula near the unregularized optimizer.  Write
\[
T:=\nabla\varphi,
\qquad
A(x):=\nabla^2\varphi(x),
\]
so that \(T_\#\mu=\nu\).  As the Brenier plan is concentrated on the graph of \(T\), it is natural to pull the second variable back by writing \(y=T(x')\).  Let
\[
a(x):=\|x\|^2-2\varphi(x),
\qquad
b(y):=\|y\|^2-2\varphi^*(y).
\]
Then \((a,b)\) is a pair of Kantorovich potentials for the (unregularized) OT problem with quadratic cost and
\begin{equation}
\|x-y\|^2-a(x)-b(y)
=
2\mathbb D(x,y),
\qquad
\mathbb D(x,y):=\varphi(x)+\varphi^*(y)-\langle x,y\rangle .
\end{equation}
We define approximate ROT potentials \((f_\varepsilon,g_\varepsilon)\) by perturbing only the first Kantorovich potential, say\footnote{For the \(p\)-ROT problem, the proof below uses a scaled convention in which the quadratic cost and Kantorovich potentials are divided by \(p\); this is only a change of normalization and accounts for the factors \(1/p\) in the \(p\)-ROT formulas.  In the present proof sketch we keep the unscaled quantities.}
\[
f_\varepsilon(x)=a(x)+\Lambda_\varepsilon(x),
\qquad
g_\varepsilon(y)=b(y).
\]
This ansatz yields an approximate pulled-back density w.r.t.\ $\mu\otimes\mu$,
\begin{equation}
\tilde h_\varepsilon(x,x')
:=
\Xi'\left(
\frac{\Lambda_\varepsilon(x)-2\mathbb D(x,T(x'))}{\varepsilon}
\right).
\end{equation}
The first marginal constraint for this approximate density would then require, for \(\mu\)-a.e.\ \(x\),
\begin{equation}
\int
\Xi'\left(
\frac{\Lambda_\varepsilon(x)-2\mathbb D(x,T(x'))}{\varepsilon}
\right)
\,d\mu(x')
=
1.
\end{equation}
In the proof, the pulled-back density is used in two different ways: as a dual ansatz for the lower bound and, after localization and freezing of coefficients, as the leading block of an admissible coupling for the upper bound.

The local profile follows by replacing the Bregman divergence with its second-order expansion.  For \(x'\) close to \(x\),
\begin{equation}
2\mathbb D(x,T(x'))
=
\|x-x'\|_{A(x)}^2+o(\|x-x'\|^2).
\end{equation}
Freezing \(f_\mu\) and \(A\) at \(x\) gives the model kernel
\begin{equation}
\label{eq:density-general}
K_{\varepsilon,x}(z)
:=
\Xi'\left(
\frac{c_\varepsilon(x)-\|z\|_{A(x)}^2}{\varepsilon}
\right),
\qquad z=x'-x,
\end{equation}
where \(c_\varepsilon(x)\) is chosen so that
\begin{equation}
\label{eq:density-eqn-to-solve}
f_\mu(x)
\int_{\mathbb R^d}K_{\varepsilon,x}(z)\,dz
=1.
\end{equation}

For the entropic and $L^p$ penalties, \eqref{eq:density-eqn-to-solve} can be solved explicitly. Indeed, for \(\Upsilon(r)=r\log r\), one has \(\Xi'(s)=e^{s-1}\).  Solving~\eqref{eq:density-eqn-to-solve} then yields the Gaussian profile
\begin{equation}
K_{\varepsilon,x}^{\rm EOT}(z)
=
\frac{\sqrt{\det A(x)}}
{f_\mu(x)(\pi\varepsilon)^{d/2}}
\exp\left(-\frac{\|z\|_{A(x)}^2}{\varepsilon}\right).
\end{equation}
For \(\Upsilon(r)=r^p\) with \(p>1\), one has \(\Xi'(s)=(s_+/p)^{1/(p-1)}\).  Equivalently, writing \(c_{\varepsilon,p}(x)=p C_{\varepsilon,p}(x)\), the normalized local kernel is
\begin{equation}
K_{\varepsilon,x}^{(p)}(z)
=
\varepsilon^{-\frac1{p-1}}
\left(
 C_{\varepsilon,p}(x)-\frac1p\|z\|_{A(x)}^2
\right)_+^{\frac1{p-1}},
\end{equation}
where
\begin{equation}
C_{\varepsilon,p}(x)
=
\left[
\frac{\sqrt{\det A(x)}}{\mathfrak B_{d,p} f_\mu(x)}
\right]^{\frac{2(p-1)}{d(p-1)+2}}
\varepsilon^{\frac{2}{d(p-1)+2}}.
\end{equation}

The same computation gives the constants in the main theorems.  By the Monge--Amp\`ere identity
\[
\det A(x)=\frac{f_\mu(x)}{f_\nu(T(x))}
\qquad \mu\text{-a.e.},
\]
the combined local transport and regularization contribution of the Gaussian profile is
\[
\varepsilon
\left[
-\frac12\log\bigl(f_\mu(x)f_\nu(T(x))\bigr)
-
\frac d2\log(\pi\varepsilon)
\right]d\mu(x),
\]
which is the integrand in \cref{theorem:main-EOT}.  Similarly, the combined local contribution of the Barenblatt profile is
\[
\varepsilon^{\frac{2}{d(p-1)+2}}
\mathfrak C_{d,p}
\bigl[f_\mu(x)f_\nu(T(x))\bigr]^{-\frac{p-1}{d(p-1)+2}}
\,d\mu(x),
\]
which yields the constant in \cref{theorem:main-p} after integration.

We next indicate how the formal profile above is converted into matching lower and upper bounds.
The lower bound is obtained from the dual problem; it is the easier (and more classical) part of the proof since the dual problem is unconstrained. By weak duality, evaluating the
dual objective at any admissible pair gives a lower bound for
\({\rm ROT}_{\varepsilon,\Upsilon}(\mu,\nu)\).  Roughly speaking, we use the perturbative ansatz $(f_\eps,g_\eps)$
described above for the admissible pair. %
More precisely, to deal with potentially unbounded marginal densities, the argument is first localized to compact sets on which
\(f_\mu\) and \(f_\nu\circ T\) are bounded above and away from zero. On such sets, the quadratic expansion localizes the dual penalty to a shrinking neighborhood of the graph of \(T\), and the calculation reduces to the Gaussian and Barenblatt integrals above. An exhaustion argument then removes the localization. In the \(p\)-ROT case this also handles the possibility \(\mathcal K(\mu,\nu,p)=\infty\). (In the EOT case, the lower bound is trivial if \(\mathcal K_{\rm EOT}(\mu,\nu)=-\infty\).)

The upper bound is the main difficulty of the proof: we need to construct an approximately optimal coupling matching the marginals $(\mu,\nu)$ exactly. We first work in source-source variables $(x,x')$ and then push forward by \(\operatorname{Id}\times T\).  If \(\pi\in\Pi(\mu,\mu)\) and \(\gamma:=(\operatorname{Id}\times T)_\#\pi\), then
\begin{equation}
\int \|x-y\|^2\,d\gamma(x,y)-{\rm OT}(\mu,\nu)
=
2\int \mathbb D(x,T(x'))\,d\pi(x,x'),
\end{equation}
thus it suffices to construct a coupling of \(\mu\) with itself whose mass is concentrated near the diagonal at the scale of the relevant local profile.

We separate the construction of this coupling into a leading part (responsible for the leading term) and a remainder. The leading part is built from an approximation to the source density from below by simple functions.  For each $n\in\N$, choose a finite partition \((E_i)_{i\in I_n}\) of ``small'' cells and corresponding constants \(a_i>0\) such that
\[
f_n:=\sum_{i\in I_n} a_i\mathbf 1_{E_i}\leq f_\mu \qquad\text{and}\qquad \lim_{n\to\infty}f_n = f_\mu.
\]
For each cell $E_i$, choose \(x_i\in E_i\), set \(A_i:=A(x_i)\), and let \(K_{i,\varepsilon}\) be the Gaussian or Barenblatt profile with \(A(x)\) and \(f_\mu(x)\) replaced by \(A_i\) and \(a_i\), which is normalized by
\[
\int_{\mathbb R^d}K_{i,\varepsilon}(z)\,dz=\frac1{a_i}
\]
according to~\eqref{eq:density-eqn-to-solve}. Define
\begin{equation}
\pi^{\varepsilon}_{n,{\rm lead}}
:=
\sum_i
 a_i^2
 K_{i,\varepsilon}(x'-x)
 \mathbf 1_{E_i}(x)\mathbf 1_{E_i}(x')
\,d\cL^d(x)d\cL^d(x').
\end{equation}
Since the kernels are even, the two marginals of \(\pi^{\varepsilon}_{n,{\rm lead}}\) coincide.  Moreover, restriction to \(E_i\times E_i\) can only decrease the total mass of the kernel, and hence
\[
(\operatorname{pr}_1)_\#\pi^{\varepsilon}_{n,{\rm lead}}
=
(\operatorname{pr}_2)_\#\pi^{\varepsilon}_{n,{\rm lead}}
\leq
f_n\,d\cL^d
\leq
\mu,
\]
where $\operatorname{pr}_j$ denotes the projection onto the $j$th marginal. Consequently,
\[
\lambda_{n,{\rm rem}}
:=
\mu-(\operatorname{pr}_1)_\#\pi^{\varepsilon}_{n,{\rm lead}}
=
\mu-(\operatorname{pr}_2)_\#\pi^{\varepsilon}_{n,{\rm lead}}
\]
is a \emph{positive} finite measure.  Coupling this residual measure with itself by any \(\pi^{\varepsilon}_{n,{\rm rem}}\in\Pi(\lambda_{n,{\rm rem}},\lambda_{n,{\rm rem}})\) gives
\[
\pi^\varepsilon_n
:=
\pi^{\varepsilon}_{n,{\rm lead}}
+
\pi^{\varepsilon}_{n,{\rm rem}}\in\Pi(\mu,\mu),
\qquad \text{and finally} \qquad
\gamma^\varepsilon_n:=(\operatorname{Id}\times T)_\#\pi^\varepsilon_n
\in\Pi(\mu,\nu).
\]
For fixed \(n\), the leading part has the asymptotics predicted by the aforementioned Gaussian or Barenblatt calculation.  Letting \(n\to\infty\) then recovers the integral constant in the main result. 

The remaining key issue is to construct the completion \(\pi^{\varepsilon}_{n,{\rm rem}}\) of the coupling so that its contribution is negligible at the relevant scale. Satisfying the hard marginal constraint is where previous approaches became technically involved and restrictive assumptions were needed: In the EOT case, \cite{pal2019difference} uses convolution arguments that are suitable for Gaussian kernels but do not adapt to the Barenblatt profiles. In the QOT case \cite{GarrizmolinaElAl.2024}, the approximate coupling is corrected by a push-forward with a delicate construction. In the present work, our guiding idea is to use the local profile only for the leading part $\pi^{\varepsilon}_{n,{\rm lead}}$ of the coupling, following the same recipe for all regularizers, and then complete the remaining marginal mass by constructing $\pi^{\varepsilon}_{n,{\rm rem}}$ using quantization (or block approximation) in a second step. For the latter, it turns out that there are two different regimes for the rigorous treatment, depending on the convexity of the regularizer. The first regime includes EOT and $p$-ROT for $1<p\leq2$ (and in particular the regularizers of the earlier works \cite{pal2019difference,GarrizmolinaElAl.2024} for EOT and QOT). A second, novel regime appears for $p>2$.

The defining feature of the first regime is that \(\Upsilon(r)=r\ell(r)\) on \((0,\infty)\), with \(\ell\) \emph{concave}. Note that this includes \(\Upsilon(r)=r\log r\) and \(\Upsilon(r)=r^p\) for \(1<p\leq2\) (but not for $p>2$).  If a finite measure of total mass \(M\) is partitioned into \(N\) cells of masses \(m_i\), and the product coupling is used on each cell, the regularization cost is controlled by
\[
\sum_i m_i^2\Upsilon\left(\frac1{m_i}\right)
=
\sum_i m_i\ell\left(\frac1{m_i}\right)
\leq
M\ell\left(\frac{N}{M}\right),
\]
where the inequality follows from Jensen's inequality due to \(\sum_i (m_i/M)(1/m_i)=N/M\) and $\ell$ being concave. This estimate will be used for $\pi^{\varepsilon}_{n,{\rm rem}}$. Note that it depends only on the total residual mass and on the number of cells, not on a lower bound for the individual \(m_i\)'s. In the proof, this fact enables a weighted quantization argument (which we do not fully detail here) under the moment assumption \(\mathcal P_{2+\beta}\) for the first regime, without requiring compact supports or lower bounds on the densities.

For the second regime, where \(p>2\), the above Jensen argument fails because \(\ell(r)=r^{p-1}\) is now convex instead of concave. Very small cells may then make \(\sum_i m_i^{2-p}\) large, so the residual mass has to be completed using cells of comparable mass. This is why the case $p>2$ in \cref{theorem:main-p} assumes bounded supports with Lipschitz boundaries and densities bounded above and below when $\mathcal K(\mu,\nu,p)$ is finite.  Under these hypotheses, cubical partitions, together with the boundary estimates for Lipschitz domains, give cells with masses comparable to \(N^{-1}\) up to constants. The regularization cost is then of order \(\varepsilon N^{p-1}\), which balances the quantization cost \(N^{-2/d}\) at the Barenblatt scale.

\section{Lower bound}\label{Section:lower}
As mentioned in the proof sketch (\cref{sec:general_recipe}), the lower bound is proved by duality. We construct test functions for the dual problem by perturbing the Kantorovich potentials at the scale suggested by the local profile. The main estimates then reduce to a local analysis near the optimal transport graph, where the Bregman divergence associated with the Brenier potential is quadratic to leading order.

\subsection{Auxiliary lemmas}\label{Section:auxiliary}

We first provide two auxiliary results for the lower bound: weak duality for the regularized problems, and elementary estimates on the Bregman divergence. The former provides the dual lower bound, while the latter identifies the local quadratic approximation that produces the sharp constant.

\subsubsection{Weak duality}

The first proposition states the standard weak duality for $p$-ROT and EOT.

\begin{proposition}\label{proposition:dual} Let $\mu,\nu\in\mathcal P_2(\mathbb R^d)$. For $p>1$ and $q=\frac{p}{p-1}$, we have
\begin{multline}
\label{eq:lq_ot_dual}
{\rm ROT}_{\varepsilon,p}(\mu,\nu)
\geq 
p\sup_{f\in L^1(\mu),\,g\in L^1(\nu)}
\bigg[
\int f\,d\mu+\int g\,d\nu \\
-\frac{1}{q\varepsilon^{\frac{1}{p-1}}}
\int
\left(
f\oplus g-\frac{\|x-y\|^2}{p}
\right)_+^q
d(\mu\otimes\nu)
\bigg].
\end{multline}
For EOT, we have 
\begin{multline*}
    {\rm EOT}_\varepsilon(\mu,\nu)
\geq
\sup_{f\in L^1(\mu),\,g\in L^1(\nu)} \bigg[\int f\,d\mu+\int g\,d\nu
\\
-
\varepsilon
\int
\exp\left(
\frac{f\oplus g-\|x-y\|^2}{\varepsilon}-1
\right)
\,d(\mu\otimes\nu)\bigg].
\end{multline*}
\end{proposition}

The proof is standard and hence omitted. 

\subsubsection{Estimates on the Bregman divergence}
Recall the potential $\varphi$ from \cref{assumption:elliptic_phi}. We define the Bregman divergence as 
\begin{equation}
\label{eq:D_func_definition}
    \mathbb{D}(x,y):=\varphi(x)+\varphi^*(y)-\langle x, y\rangle .
\end{equation}
 The following lemma is \cite[Lemma 2.2]{GarrizmolinaElAl.2024}. Qualitatively, it states that the Bregman divergence~\(\mathbb D\) between $x$ and $y$ is comparable to the squared Euclidean distance between $x$ and the pull-back of $y$ along the optimal transport map.
\begin{lemma}
\label{lemma:Bregman_ptwisebd}
    Let \ref{assumption:elliptic_phi} hold. Then 
    $$
    \frac{\sigma_m(\varphi)}{2}\|x-x'\|^2\leq \mathbb{D}(x,\, \nabla \varphi(x'))\leq \frac{\sigma_M(\varphi)}{2}\|x-x'\|^2
    $$
and for every compact set $K\subset\R^d$ there exists $\omega:\R_+\to \R_+$ such that $\lim_{t\to0}\omega(t)= 0$ and
\begin{equation}
    \label{eq:bregman_is_norm}
    \left|2\mathbb{D}(x,\, \nabla \varphi(x'))-\|x-x'\|^2_{\nabla^2\varphi(x)}\right|\leq \omega(\|x-x'\|) \|x-x'\|^{2}, \qquad x,x'\in K.
\end{equation}
\end{lemma}

We shall also use the following integrated form of the Bregman divergence.
\begin{lemma}\label{lemma:integrated-bregman-identity}
Let \(\pi\in\Pi(\mu,\mu)\) and set
\(\gamma:=(\operatorname{Id}\times\nabla\varphi)_\#\pi\in\Pi(\mu,\nu)\). Then
\[
\int \|x-y\|^2\,d\gamma(x,y)-{\rm OT}(\mu,\nu)
=
2\int \mathbb D(x,\nabla\varphi(x'))\,d\pi(x,x').
\]
\end{lemma}
\begin{proof}
Let
\[
a(x):=\|x\|^2-2\varphi(x),
\qquad
b(y):=\|y\|^2-2\varphi^*(y).
\]
Then \((a,b)\) is a pair of Kantorovich potentials for the cost
\(\|x-y\|^2\), and therefore
\[
\int a\,d\mu+\int b\,d\nu={\rm OT}(\mu,\nu).
\]
Moreover,
\[
\|x-y\|^2-a(x)-b(y)
=
2\bigl(\varphi(x)+\varphi^*(y)-\langle x,y\rangle\bigr)
=
2\mathbb D(x,y).
\]
Since \(\gamma\in\Pi(\mu,\nu)\), it follows that
\[
\int \|x-y\|^2\,d\gamma-{\rm OT}(\mu,\nu)
=
\int\bigl(\|x-y\|^2-a(x)-b(y)\bigr)\,d\gamma
=
2\int \mathbb D(x,y)\,d\gamma.
\]
Finally, using \(\gamma=(\operatorname{Id}\times\nabla\varphi)_\#\pi\) gives the claimed identity.
\end{proof}

\proofreadhere

\subsection{Proof of the lower bound for $p$-ROT}

Next, we derive the lower bound for $p$-ROT; the analogue for EOT is proved in the subsequent subsection.

We begin by recording a standard change of variables for ease of reference.

\begin{lemma}[Monge--Amp\`ere change of variables]
\label{lemma:MA-change-of-variables}
Let \ref{assumption:hoelder_densities} and
\ref{assumption:elliptic_phi} hold, and set \(T:=\nabla\varphi\). Then \(T\)
is a \(C^1\)-diffeomorphism of \(\mathbb R^d\) onto itself and
\[
\sigma_m(\varphi)\|x-x'\|
\le
\|T(x)-T(x')\|
\le
\sigma_M(\varphi)\|x-x'\|,
\qquad x,x'\in\mathbb R^d .
\]
Moreover,
\begin{equation}
\label{eq:MA-density-identity}
f_\mu(x)
=
f_\nu(T(x))\det\nabla^2\varphi(x)
\qquad\text{for }\cL^d\text{-a.e. }x .
\end{equation}

Let \(\Psi(x,x')=(x,T(x'))\). If
\(\lambda\) is a nonnegative measure on $\mathbb R^d\times\mathbb R^d$ and
\(\widehat\lambda:=\Psi_\#\lambda\), then
$\lambda\ll \mu\otimes\mu$ if and only if $\widehat\lambda\ll \mu\otimes\nu$, and in that case,
\begin{equation}
\label{eq:RN-density-under-T}
\frac{d\widehat\lambda}{d(\mu\otimes\nu)}(\Psi(x,x'))
=
\frac{d\lambda}{d(\mu\otimes\mu)}(x,x')
\qquad
(\mu\otimes\mu)\text{-a.e.}
\end{equation}
In particular, the \(L^p\) and relative entropy penalties are invariant under \(\Psi\).
\end{lemma}

Next, we provide an auxiliary lemma for the main proof.

\begin{lemma}\label{lem:lower-bound-lebesgue-diff}
Let \(p>1\) and set $q=\frac{p}{p-1}$. 
Let \(K\subset \R^d\) be compact and let \(E\subset K\) be measurable. Assume that,
for some \(\gamma\in(0,1)\),
\[
\gamma\leq f_\mu(x)\leq \gamma^{-1},
\qquad
\gamma\leq f_\nu(\nabla\varphi(x))\leq \gamma^{-1},
\qquad x\in E.
\]
For \(x\in E\), set
\[
C_{\varepsilon,p}(x)
:=
\varepsilon^{\frac{2}{d(p-1)+2}} \left[
\frac{\sqrt{\det \nabla^2\varphi(x)}}{\mathfrak B_{d,p}f_\mu(x)}
\right]^{\frac{2(p-1)}{d(p-1)+2}}
\]
and
\[
L_{\varepsilon,E}'
:=
\varepsilon^{-\frac1{p-1}}
\int_E
\int_{\mathbb R^d}
\left[
C_{\varepsilon,p}(x)
-
\frac{2}{p}\mathbb{D}(x,\nabla\varphi(x'))
\right]_+^q
\,d\mu(x')\,d\mu(x).
\]
Then
\[
\lim_{\varepsilon\downarrow0}
\varepsilon^{-\frac{2}{d(p-1)+2}}
L_{\varepsilon,E}'
=
\frac{2p}{d(p-1)+2p}\,
\mathfrak B_{d,p}^{-\frac{2(p-1)}{d(p-1)+2}}
\int_E
\left[
f_\nu(\nabla\varphi(x))f_\mu(x)
\right]^{-\frac{p-1}{d(p-1)+2}}
\,d\mu(x).
\]
\end{lemma}

\begin{proof} 
Since \(x\in E\) and \(f_\mu,f_\nu\circ\nabla\varphi\) are bounded above and below on \(E\),
and since \(\nabla^2\varphi(x)\) is uniformly elliptic, there exist constants
$
0<c_\gamma<C_\gamma<\infty
$
such that
\begin{equation}
    \label{eq:bound-Ceps-p}
c_\gamma \varepsilon^{\frac{2}{d(p-1)+2}}\leq  C_{\varepsilon,p}(x)\leq C_\gamma \varepsilon^{\frac{2}{d(p-1)+2}},
\qquad x\in E.
\end{equation} 
We first record the localization of the Barenblatt profile. If
\[
\left[
C_{\varepsilon,p}(x)
-
\frac{2}{p}\mathbb{D}(x,\nabla\varphi(x'))
\right]_+>0,
\]
then $\frac{2}{p}\mathbb{D}(x,\nabla\varphi(x'))\leq C_{\varepsilon,p}(x).$ 
By \Cref{lemma:Bregman_ptwisebd},
\[
\mathbb{D}(x,\nabla\varphi(x'))\geq \frac{\sigma_m(\varphi)}{2}\|x-x'\|^2.
\]
Hence $\|x-x'\|
\leq
C \varepsilon^{\frac{1}{d(p-1)+2}},$ 
with \(C\) depending only on \(p,\gamma,\sigma_m(\varphi),\sigma_M(\varphi)\). Therefore there exists
\(R_\gamma>0\), independent of \(\varepsilon\), such that the integrand in $L'_{\eps,E}$ is supported in
\begin{equation}
    \label{localization-lower}
    \left\{ (x,x'): \| x-x'\|\leq R_\gamma \varepsilon^{\frac{1}{d(p-1)+2}} \right\}.
\end{equation}
For \(x\in E\), define
\[
J_\varepsilon(x)
:=
\varepsilon^{-\frac{2}{d(p-1)+2}}
\varepsilon^{-\frac1{p-1}}
\int_{\mathbb R^d}
\left[
C_{\varepsilon,p}(x)
-
\frac{2}{p}\mathbb{D}(x,\nabla\varphi(x'))
\right]_+^q
\,d\mu(x'),
\]
and note that
$$
\varepsilon^{-\frac{2}{d(p-1)+2}}
L_{\varepsilon,E}' = \int_E J_\varepsilon(x) d\mu(x).
$$
Changing variables $x'=x+\varepsilon^{\frac{1}{d(p-1)+2}}  z,$ 
we obtain
\[
J_\varepsilon(x)
=
\int_{\mathbb R^d}
H_\varepsilon(x,z) f_\mu(x+\varepsilon^{\frac{1}{d(p-1)+2}}  z)\,dz,
\]
where
\[
H_\varepsilon(x,z)
:=
 \left[
{\varepsilon^{-\frac{2}{d(p-1)+2}} }C_{\varepsilon,p}(x)-
{\varepsilon^{-\frac{2}{d(p-1)+2}} }\frac{2}{p}
{\mathbb{D}(x,\nabla\varphi(x+\varepsilon^{\frac{1}{d(p-1)+2}}  z))}
\right]_+^q.
\]
We have $H_\varepsilon(x,z)=0$ for  $\|z\|>R_\gamma$ by~\eqref{localization-lower}. 
Moreover, \eqref{eq:bound-Ceps-p} and \Cref{lemma:Bregman_ptwisebd} yield
\begin{equation}
    \label{eq:bound-Heps}
    0\leq H_\varepsilon(x,z)
\leq C_\gamma \mathbf 1_{\ball(0,R_\gamma)}(z).
\end{equation}
We now identify the pointwise limit. By \Cref{lemma:Bregman_ptwisebd}, applied on a compact neighborhood of
\(K\), we have uniformly for \(x\in E\) and \(\|z\|\leq R_\gamma\),
\[
\varepsilon^{-\frac{2}{d(p-1)+2}}\frac{2}{p}
{\mathbb{D}(x,\nabla\varphi(x+\varepsilon^{\frac{1}{d(p-1)+2}}  z))}
=
\frac1p \|z\|_{\nabla^2\varphi(x)}^2+o_\gamma(1).
\]
Therefore
\[
H_\varepsilon(x,z)
\to
H_0(x,z)
:=
\left[
\left[
\frac{\sqrt{\det \nabla^2\varphi(x)}}{\mathfrak B_{d,p}f_\mu(x)}
\right]^{\frac{2(p-1)}{d(p-1)+2}}
-
\frac1p\|z\|_{\nabla^2\varphi(x)}^2
\right]_+^q
\]
for \(\mathcal L^d\)-a.e.\ \(x\in E\) and every \(z\in\mathbb R^d\), and the convergence is dominated by
\(C_\gamma\mathbf 1_{\ball(0,R_\gamma)}\).

Let \(x\in E\) be a Lebesgue point of \(f_\mu\). By the Lebesgue differentiation theorem (see \cite[Theorem~3.21]{folland1999real}),
for every fixed \(R>0\),
\begin{equation}
    \label{lebesgue-diff-lower-p}
    \int_{\ball(0,R)}
|f_\mu(x+\varepsilon^{\frac{1}{d(p-1)+2}}  z)-f_\mu(x)|\,dz
\to 0.
\end{equation} 
 Hence
 \begin{multline*}
     \left|
J_\varepsilon(x)
-
f_\mu(x)\int_{\mathbb R^d}H_0(x,z)\,dz
\right|
\\
\leq
\int_{\mathbb R^d}
H_\varepsilon(x,z)
|f_\mu(x+\varepsilon^{\frac{1}{d(p-1)+2}}  z)-f_\mu(x)|\,dz
\\
+
f_\mu(x)
\int_{\mathbb R^d}
|H_\varepsilon(x,z)-H_0(x,z)|\,dz.
 \end{multline*}
The first term tends to \(0\) by \eqref{eq:bound-Heps} and \eqref{lebesgue-diff-lower-p}. The second term tends to \(0\)
by dominated convergence. Therefore, as $\varepsilon\downarrow 0$,
\[
J_\varepsilon(x)
\to 
J_0(x)
:=
f_\mu(x)
\int_{\mathbb R^d}
\left[
\left[
\frac{\sqrt{\det \nabla^2\varphi(x)}}{\mathfrak B_{d,p}f_\mu(x)}
\right]^{\frac{2(p-1)}{d(p-1)+2}}-\frac1p\|z\|_{\nabla^2\varphi(x)}^2
\right]_+^q
\,dz
\]
for \(\cL^d\)-a.e.~\(x\in E\).

We next upgrade this pointwise convergence to convergence after integration over \(E\).
It is enough to prove convergence in \(L^1(E)\), because
\(f_\mu\leq \gamma^{-1}\) on \(E\). We use a truncation argument. For \(M>0\), set $f_\mu^M:=f_\mu\wedge M$, and define 

\[
J_\varepsilon^M(x)
:=
\int_{\mathbb R^d}
H_\varepsilon(x,z) f_\mu^M(x+\varepsilon^{\frac{1}{d(p-1)+2}}  z)\,dz.
\]
The same argument as above gives $J_\varepsilon^M(x)
\to 
f_\mu^M(x)
\int_{\mathbb R^d}H_0(x,z)\,dz$, 
for \(\cL^d\)-a.e.~\(x\in E\). Moreover, $0\leq J_\varepsilon^M(x)
\leq
C_\gamma M |\ball(0,R_\gamma)|.$
Since \(E\) has finite Lebesgue measure, the dominated convergence theorem implies that, as $\varepsilon\downarrow 0$,
\[
J_\varepsilon^M
\to 
f_\mu^M
\int_{\mathbb R^d}H_0(\cdot,z)\,dz
\qquad
\text{in }L^1(E).
\]
It remains to remove the truncation. We have
\begin{align}
\int_E |J_\varepsilon(x)-J_\varepsilon^M(x)|\,dx
&\leq
C_\gamma
\int_E
\int_{\ball(0,R_\gamma)}
\bigl(f_\mu-f_\mu^M\bigr)(x+\varepsilon^{\frac{1}{d(p-1)+2}}  z)
\,dz\,dx
\\
&=
C_\gamma
\int_{\ball(0,R_\gamma)}
\int_{E+\varepsilon^{\frac{1}{d(p-1)+2}}  z}
\bigl(f_\mu-f_\mu^M\bigr)(y)
\,dy\,dz
\\
&\leq
C_\gamma |\mathbb{B}(0,R_\gamma)|
\int_{\mathbb R^d}
\bigl(f_\mu-f_\mu^M\bigr)(y)\,dy.
\end{align}
Similarly,
\[
\int_E
\left|
f_\mu(x)-f_\mu^M(x)
\right|
\int_{\mathbb R^d}H_0(x,z)\,dz\,dx
\leq
C_\gamma |\ball(0,R_\gamma)|
\int_{\mathbb R^d}
\bigl(f_\mu-f_\mu^M\bigr)(y)\,dy.
\]
Since \(f_\mu\in L^1(\mathbb R^d)\), the last quantity tends to \(0\) as \(M\to\infty\).
Therefore,  $J_\varepsilon
\to  J_0$ in $L^1(E)$. 
As \(f_\mu\leq\gamma^{-1}\) on \(E\), this also gives
\[
\int_E J_\varepsilon(x)\,d\mu(x)
\to 
\int_E J_0(x)\,d\mu(x),
\]
as $\varepsilon\downarrow 0$. 

We now compute the limit. By \Cref{lemma:integrals} with \(\varepsilon=1\) and
\(c=\left[
\frac{\sqrt{\det \nabla^2\varphi(x)}}{\mathfrak B_{d,p}f_\mu(x)}
\right]^{\frac{2(p-1)}{d(p-1)+2}}\),
\[
\int_{\mathbb R^d}
\left[
\left[
\frac{\sqrt{\det \nabla^2\varphi(x)}}{\mathfrak B_{d,p}f_\mu(x)}
\right]^{\frac{2(p-1)}{d(p-1)+2}}-\frac1p\|z\|_{\nabla^2\varphi(x)}^2
\right]_+^{\frac1{p-1}}
dz
=
\frac{1}{f_\mu(x)}.
\]
Moreover, using the beta-function identity ${
B\left(\frac d2,\frac{2p-1}{p-1}\right)
}
=
\frac{2p}{d(p-1)+2p}B\left(\frac d2,\frac{p}{p-1}\right)$, 
we get
\[
\int_{\mathbb R^d}
\left[
\left[
\frac{\sqrt{\det \nabla^2\varphi(x)}}{\mathfrak B_{d,p}f_\mu(x)}
\right]^{\frac{2(p-1)}{d(p-1)+2}}
\!\!\!\!-\frac1p\|z\|_{\nabla^2\varphi(x)}^2
\right]_+^q \!\!\!\!
dz
=
\frac{2p}{d(p-1)+2p}
\frac{\left[
\frac{\sqrt{\det \nabla^2\varphi(x)}}{\mathfrak B_{d,p}f_\mu(x)}
\right]^{\frac{2(p-1)}{d(p-1)+2}}}{f_\mu(x)}.
\]
Therefore
\[
J_0(x)
=
\frac{2p}{d(p-1)+2p} 
\left[
\frac{\sqrt{\det \nabla^2\varphi(x)}}{\mathfrak B_{d,p}f_\mu(x)}
\right]^{\frac{2(p-1)}{d(p-1)+2}}.
\] 
By \eqref{eq:MA-density-identity},
\[
\left[
\frac{\sqrt{\det \nabla^2\varphi(x)}}{\mathfrak B_{d,p}f_\mu(x)}
\right]^{\frac{2(p-1)}{d(p-1)+2}}
=
\mathfrak B_{d,p}^{-\frac{2(p-1)}{d(p-1)+2}}
\left[
f_\nu(\nabla\varphi(x))f_\mu(x)
\right]^{-\frac{p-1}{d(p-1)+2}},
\]
and hence
\[
J_0(x)
=
\frac{2p}{d(p-1)+2p}
\mathfrak B_{d,p}^{-\frac{2(p-1)}{d(p-1)+2}}
\left[
f_\nu(\nabla\varphi(x))f_\mu(x)
\right]^{-\frac{p-1}{d(p-1)+2}}.
\]
Consequently,
\[
\lim_{\varepsilon\downarrow0}
\varepsilon^{-\frac{2}{d(p-1)+2}}
L_{\varepsilon,E}'
=
\frac{2p}{d(p-1)+2p}\,
\mathfrak B_{d,p}^{-\frac{2(p-1)}{d(p-1)+2}}
\int_E
\left[
f_\nu(\nabla\varphi(x))f_\mu(x)
\right]^{-\frac{p-1}{d(p-1)+2}}
\,d\mu(x),
\]
completing the proof.
\end{proof}

\proofreadhere

We can now state and prove the lower bound. 

\begin{theorem}[Lower bound for $p$-ROT]
\label{thm:lower_bound}
Fix $p>1$ and assume \ref{assumption:hoelder_densities}, \ref{assumption:elliptic_phi}. Then
$$ \liminf_{\varepsilon \to 0^+}\frac{{\rm ROT}_{\eps,p}(\mu,\nu)-{\rm OT}(\mu,\nu)}{\varepsilon^{\frac{2}{d(p-1)+2}}}\geq \mathfrak C_{d,p}\cdot\mathcal{K}(\mu,\nu,p)\in[0,\infty].$$
\end{theorem}

\begin{proof}%
Set \(q:=\frac{p}{p-1}\) and
\[
\Theta(x):=
\left[
f_\nu(\nabla\varphi(x))f_\mu(x)
\right]^{-\frac{p-1}{d(p-1)+2}}.
\]
We first prove that for every compact set \(K\subset\operatorname{supp}\mu\),
\begin{align}\label{eq:compact-lower-bound}
\liminf_{\varepsilon\downarrow0}F_{\varepsilon,p}
\ge
\mathfrak C_{d,p}\int_K\Theta(x)\,d\mu(x),
\end{align}
where
\[
F_{\varepsilon,p}
:=
\frac{
\operatorname{ROT}_{\varepsilon,p}(\mu,\nu)-\operatorname{OT}(\mu,\nu)
}{
\varepsilon^{\frac{2}{d(p-1)+2}}
}.
\]
Fix such a compact set \(K\).  For \(\gamma\in(0,1)\), define
\[
K_\gamma
:=
K\cap
\left\{x:
\gamma\leq f_\mu(x)\leq \gamma^{-1}
\ {\rm and}\ 
\gamma\leq (f_\nu\circ \nabla\varphi)(x)\leq \gamma^{-1}
\right\}.
\]
Since \(f_\mu\) and \(f_\nu\circ\nabla\varphi\) are finite and strictly positive \(\mu\)-a.e.\ on \(\operatorname{supp}\mu\), we have \(\mathbf 1_{K_\gamma}\uparrow \mathbf 1_K\) \(\mu\)-a.e.\ as \(\gamma\downarrow0\).  For \(x\in K_\gamma\), set
\[
C_{\varepsilon,p}(x)
:=
\varepsilon^{\frac{2}{d(p-1)+2}}
\left[
\frac{\sqrt{\det \nabla^2\varphi(x)}}{\mathfrak B_{d,p}f_\mu(x)}
\right]^{\frac{2(p-1)}{d(p-1)+2}}.
\]
Define
\[
\Gamma_{\varepsilon,p}(a,b):=
\int a\,d\mu+\int b\,d\nu
-\frac{1}{q\varepsilon^{\frac{1}{p-1}}}
\int
\left(
a\oplus b-\frac{\|x-y\|^2}{p}
\right)_+^q
d(\mu\otimes\nu).
\]
We choose the dual test functions
\[
a(x)
:=
\frac{\|x\|^2}{p}
-
\frac{2}{p}\varphi(x)
+
C_{\varepsilon,p}(x)\mathbf 1_{K_\gamma}(x),\qquad b(y)
:=
\frac{\|y\|^2}{p}
-
\frac{2}{p}\varphi^*(y),
\] 
which satisfy \(a\in L^1(\mu)\) and \(b\in L^1(\nu)\) due to \ref{assumption:hoelder_densities}, \ref{assumption:elliptic_phi}. By weak duality (\Cref{proposition:dual}),
\[
F_{\varepsilon,p}
\geq
p\,\varepsilon^{-\frac{2}{d(p-1)+2}}
\left[
\Gamma_{\varepsilon,p}(a,b)-\frac1p\operatorname{OT}(\mu,\nu)
\right].
\]
Using Kantorovich duality (e.g., \cite[Chapter~5]{Villani2008}) and the identity
\[
\frac{\|x\|^2}{p}
-\frac{2}{p}\varphi(x)
+
\frac{\|y\|^2}{p}
-\frac{2}{p}\varphi^*(y)
-
\frac{\|x-y\|^2}{p}
=
-\frac{2}{p}\mathbb{D}(x,y),
\]
we get
\begin{equation}
    \label{eq:decomposition-lower}
    \Gamma_{\varepsilon,p}(a,b)-\frac1p\operatorname{OT}(\mu,\nu)
=
\int_{K_\gamma} C_{\varepsilon,p}(x)\,d\mu(x)
-
\frac1q L_{\varepsilon,K_\gamma}',
\end{equation}
where
\[
L_{\varepsilon,K_\gamma}'
:=
\varepsilon^{-\frac1{p-1}}
\int_{K_\gamma}
\int_{\mathbb R^d}
\left[
C_{\varepsilon,p}(x)
-
\frac{2}{p}\mathbb{D}(x,\nabla\varphi(x'))
\right]_+^q
\,d\mu(x')\,d\mu(x).
\]

The integral in \eqref{eq:decomposition-lower} is explicit. Indeed, by
\eqref{eq:MA-density-identity} and the definition of \(\Theta\),
\[
\varepsilon^{-\frac{2}{d(p-1)+2}}C_{\varepsilon,p}(x)
=
\mathfrak B_{d,p}^{-\frac{2(p-1)}{d(p-1)+2}}
\Theta(x).
\]
Therefore
\[
\varepsilon^{-\frac{2}{d(p-1)+2}}
\int_{K_\gamma}C_{\varepsilon,p}(x)\,d\mu(x)
=
\mathfrak B_{d,p}^{-\frac{2(p-1)}{d(p-1)+2}}
\int_{K_\gamma}\Theta(x)\,d\mu(x).
\]

For the second term in \eqref{eq:decomposition-lower}, \Cref{lem:lower-bound-lebesgue-diff} gives
\[
\lim_{\varepsilon\downarrow0}
\varepsilon^{-\frac{2}{d(p-1)+2}}L_{\varepsilon,K_\gamma}'
=
\frac{2p}{d(p-1)+2p}\,
\mathfrak B_{d,p}^{-\frac{2(p-1)}{d(p-1)+2}}
\int_{K_\gamma}\Theta(x)\,d\mu(x).
\]
Consequently,
\[
\begin{aligned}
\liminf_{\varepsilon\downarrow0}F_{\varepsilon,p}
&\geq
p
\left[
1
-
\frac1q
\frac{2p}{d(p-1)+2p}
\right]
\mathfrak B_{d,p}^{-\frac{2(p-1)}{d(p-1)+2}}
\int_{K_\gamma}\Theta(x)\,d\mu(x)
\\
&=
\frac{p(d(p-1)+2)}{d(p-1)+2p}
\mathfrak B_{d,p}^{-\frac{2(p-1)}{d(p-1)+2}}
\int_{K_\gamma}\Theta(x)\,d\mu(x)
=
\mathfrak{C}_{d,p}
\int_{K_\gamma}\Theta(x)\,d\mu(x).
\end{aligned}
\]
Letting \(\gamma\downarrow0\) and using the monotone convergence theorem proves the compact-set lower bound.

\emph{Case \(\mathcal{K}(\mu,\nu,p)<\infty\).} Fix \(\delta>0\). Since \(\Theta\,d\mu\) is a finite Borel measure on \(\mathbb R^d\), inner regularity gives a compact set \(K\subset\operatorname{supp}\mu\) such that
\[
\int_{\R^d\setminus K}\Theta(x)\,d\mu(x)\leq \delta.
\]
The lower bound~\eqref{eq:compact-lower-bound} yields \(\liminf_{\varepsilon\downarrow0}F_{\varepsilon,p}\geq
\mathfrak{C}_{d,p}\bigl(\mathcal{K}(\mu,\nu,p)-\delta\bigr)\), and letting \(\delta\downarrow0\) then proves
\[
\liminf_{\varepsilon\downarrow0}
\frac{
\operatorname{ROT}_{\varepsilon,p}(\mu,\nu)-\operatorname{OT}(\mu,\nu)
}{
\varepsilon^{2/(d(p-1)+2)}
}
\geq
\mathfrak{C}_{d,p}\mathcal{K}(\mu,\nu,p).
\]

\emph{Case \(\mathcal{K}(\mu,\nu,p)=\infty\).} Fix \(\delta>0\). Applying inner regularity to
\((\Theta\wedge M)\,d\mu\) and choosing \(M\) sufficiently large, we may choose
a compact set \(K\subset\operatorname{supp}\mu\) such that
\[
\int_K\Theta(x)\,d\mu(x)\geq \delta^{-1}.
\]
The lower bound~\eqref{eq:compact-lower-bound} then yields \(\liminf_{\varepsilon\downarrow0}F_{\varepsilon,p}
\geq
\mathfrak{C}_{d,p}\delta^{-1}\). Thus \(\liminf_{\varepsilon\downarrow0}F_{\varepsilon,p}=+\infty\) by letting \(\delta\downarrow0\), completing the proof.
\end{proof}

\subsection{Proof of the lower bound for EOT}

It remains to prove the analogous lower bound for EOT, which is simpler than for $p$-ROT.

\begin{theorem}[Lower bound for EOT]\label{thm:eot-lower-bound}
Let \ref{assumption:hoelder_densities}, \ref{assumption:elliptic_phi} hold.  Then
\[
\liminf_{\varepsilon\downarrow0}
\frac{
{\rm EOT}_\varepsilon(\mu,\nu)
-
{\rm OT}(\mu,\nu)
+
\frac d2\,\varepsilon\log(\pi\varepsilon)
}{\varepsilon}
\geq
\mathcal K_{\mathrm{EOT}}(\mu,\nu) \in [-\infty,\infty).
\]
\end{theorem}

\begin{proof}
If \(\mathcal K_{\rm EOT}(\mu,\nu)=-\infty\), there is nothing to prove. Assume \(\mathcal K_{\rm EOT}(\mu,\nu)>-\infty\) and note that this implies ${\rm Ent}(\mu)<\infty$ and ${\rm Ent}(\nu)<\infty$. Set \(T:=\nabla\varphi\) and \(A(x):=\nabla^2\varphi(x)\).  Define 
\[
C_\varepsilon(x)
:=
\varepsilon
\left[
1+\frac12\log\det A(x)-\log f_\mu(x)
-\frac d2\log(\pi\varepsilon)
\right].
\]
We use the entropic weak duality from \Cref{proposition:dual} with
\[
a_\varepsilon(x):=\|x\|^2-2\varphi(x)+C_\varepsilon(x),
\qquad
b(y):=\|y\|^2-2\varphi^*(y).
\]
Then \(a_\varepsilon\in L^1(\mu)\) and \(b\in L^1(\nu)\) as \(\mu,\nu\in\mathcal P_2(\mathbb R^d)\), Assumption~\ref{assumption:elliptic_phi} gives quadratic growth of \(\varphi\) and \(\varphi^*\), and \(\log f_\mu\in L^1(\mu)\).

Recall
\[
\mathbb D(x,y):=\varphi(x)+\varphi^*(y)-\langle x,y\rangle .
\]
Then
\[
a_\varepsilon(x)+b(y)-\|x-y\|^2
=
C_\varepsilon(x)-2\mathbb D(x,y),
\]
and Kantorovich duality gives
\[
\int\bigl(\|x\|^2-2\varphi(x)\bigr)\,d\mu(x)
+
\int\bigl(\|y\|^2-2\varphi^*(y)\bigr)\,d\nu(y)
=
{\rm OT}(\mu,\nu).
\]
Hence
\begin{equation}
\label{eq:development-EOT-lower}
{\rm EOT}_\varepsilon(\mu,\nu)-{\rm OT}(\mu,\nu)
\geq
\int C_\varepsilon\,d\mu-\varepsilon L_\varepsilon,
\end{equation}
where
\[
L_\varepsilon
:=
e^{-1}
\int_{\mathbb R^d}\int_{\mathbb R^d}
\exp\left(
\frac{C_\varepsilon(x)-2\mathbb D(x,y)}{\varepsilon}
\right)
\,d\nu(y)d\mu(x).
\]

We next prove that
\[
\limsup_{\varepsilon\downarrow0}L_\varepsilon\le1.
\]
Since \(T_\#\mu=\nu\), changing variables \(y=T(x')\) gives
\[
L_\varepsilon
=
e^{-1}
\int_{\mathbb R^d}\int_{\mathbb R^d}
\exp\left(
\frac{C_\varepsilon(x)-2\mathbb D(x,T(x'))}{\varepsilon}
\right)
\,d\mu(x')d\mu(x).
\]
By the definition of \(C_\varepsilon\),
\[
e^{-1}\exp\left(\frac{C_\varepsilon(x)}{\varepsilon}\right)f_\mu(x)
=
\frac{\sqrt{\det A(x)}}{(\pi\varepsilon)^{d/2}}
\qquad
\text{for }\cL^d\text{-a.e.\ }x\text{ with }f_\mu(x)>0.
\]
Consequently,
\[
L_\varepsilon
\le
\widetilde L_\varepsilon,
\]
where
\[
\widetilde L_\varepsilon
:=
\frac1{(\pi\varepsilon)^{d/2}}
\int_{\mathbb R^d}\int_{\mathbb R^d}
\sqrt{\det A(x)}
\exp\left(
-\frac{2\mathbb D(x,T(x'))}{\varepsilon}
\right)
f_\mu(x')\,dx'\,dx .
\]
It remains to show that \(\widetilde L_\varepsilon\to1\).

By Fubini,
\[
\widetilde L_\varepsilon
=
\int_{\mathbb R^d}S_\varepsilon(x')f_\mu(x')\,dx',
\]
where
\[
S_\varepsilon(x')
:=
\frac1{(\pi\varepsilon)^{d/2}}
\int_{\mathbb R^d}
\sqrt{\det A(x)}
\exp\left(
-\frac{2\mathbb D(x,T(x'))}{\varepsilon}
\right)
\,dx.
\]
Uniform ellipticity gives
\[
\mathbb D(x,T(x'))
=
\varphi(x)-\varphi(x')-\langle\nabla\varphi(x'),x-x'\rangle
\geq
\frac{\sigma_m(\varphi)}2\|x-x'\|^2,
\]
and \(\sqrt{\det A(x)}\le \sigma_M(\varphi)^{d/2}\). Hence
\[
0\le S_\varepsilon(x')
\le
\frac{\sigma_M(\varphi)^{d/2}}{(\pi\varepsilon)^{d/2}}
\int_{\mathbb R^d}
\exp\left(
-\frac{\sigma_m(\varphi)\|x-x'\|^2}{\varepsilon}
\right)
\,dx
=
\left(\frac{\sigma_M(\varphi)}{\sigma_m(\varphi)}\right)^{d/2}.
\]
For fixed \(x'\), changing variables \(x=x'+\sqrt\varepsilon z\) yields
\[
S_\varepsilon(x')
=
\frac1{\pi^{d/2}}
\int_{\mathbb R^d}
\sqrt{\det A(x'+\sqrt\varepsilon z)}
\exp\left(
-\frac{2\mathbb D(x'+\sqrt\varepsilon z,T(x'))}{\varepsilon}
\right)
\,dz.
\]
Since \(\varphi\in C^2\),
\[
\frac{2\mathbb D(x'+\sqrt\varepsilon z,T(x'))}{\varepsilon}
\longrightarrow
\|z\|_{A(x')}^2
\qquad
\text{for every fixed }z\in\mathbb R^d.
\]
The preceding ellipticity bound gives the integrable upper bound
\[
\sqrt{\det A(x'+\sqrt\varepsilon z)}
\exp\left(
-\frac{2\mathbb D(x'+\sqrt\varepsilon z,T(x'))}{\varepsilon}
\right)
\le
\sigma_M(\varphi)^{d/2}e^{-\sigma_m(\varphi)\|z\|^2}.
\]
Thus, by dominated convergence,
\[
S_\varepsilon(x')
\to
\frac{\sqrt{\det A(x')}}{\pi^{d/2}}
\int_{\mathbb R^d}e^{-\|z\|_{A(x')}^2}\,dz
=
1.
\]
Using the uniform bound on \(S_\varepsilon\) and \(f_\mu\in L^1(\mathbb R^d)\), another application of dominated convergence gives \(\widetilde L_\varepsilon\to1\). Therefore
\[
\limsup_{\varepsilon\downarrow0}L_\varepsilon\le1.
\]

Finally,
\[
\int \frac{C_\varepsilon(x)}{\varepsilon}\,d\mu(x)
=
1+
\int
\left[
\frac12\log\det A(x)-\log f_\mu(x)
\right]
d\mu(x)
-\frac d2\log(\pi\varepsilon).
\]
Combining this identity with \eqref{eq:development-EOT-lower} gives
\[
\begin{aligned}
\frac{
{\rm EOT}_\varepsilon(\mu,\nu)
-
{\rm OT}(\mu,\nu)
+
\frac d2\,\varepsilon\log(\pi\varepsilon)
}{\varepsilon}
&\geq
1+
\int
\left[
\frac12\log\det A(x)-\log f_\mu(x)
\right]
d\mu(x)
-
L_\varepsilon .
\end{aligned}
\]
Taking $\liminf$ and using \(\limsup_{\varepsilon\downarrow0}L_\varepsilon\le1\), we obtain
\[
\liminf_{\varepsilon\downarrow0}
\frac{
{\rm EOT}_\varepsilon(\mu,\nu)
-
{\rm OT}(\mu,\nu)
+
\frac d2\,\varepsilon\log(\pi\varepsilon)
}{\varepsilon}
\geq
\int
\left[
\frac12\log\det A(x)-\log f_\mu(x)
\right]
d\mu(x).
\]
It remains to see that the right-hand side equals
\(\mathcal K_{\mathrm{EOT}}(\mu,\nu)\). By \eqref{eq:MA-density-identity},
\[
\frac12\log\det A(x)-\log f_\mu(x)
=
-\frac12\log f_\mu(x)-\frac12\log f_\nu(T(x)),
\]
and now the claim follows from \(T_\#\mu=\nu\).
\end{proof}

\section{Upper bound}\label{Section:main-upper}

In this section, we prove the upper bounds for $p$-ROT and EOT, and hence complete the proofs of the main results. Note that for $p$-ROT, the upper bound is trivial if \(\mathcal K(\mu,\nu,p)=\infty\), so that we may assume \(\mathcal K(\mu,\nu,p)<\infty\).

We proceed through the primal problem: for every coupling $\pi\in\Pi(\mu,\nu)$, the primal $p$-ROT objective satisfies
$$
H_{\varepsilon,p}[\pi]:=\int \|x-y\|^2 d\pi +\varepsilon\left\|\frac{d\pi}{d(\mu\otimes \nu)}\right\|_{L^p(\mu\otimes \nu)}^p \geq {\rm ROT}_{\eps,p}(\mu,\nu). 
$$
Hence we aim to construct a family $(\pi_\varepsilon)_{\varepsilon >0}$ of couplings so that $H_{\varepsilon,p}[\pi_\varepsilon]$ yields the claimed upper bound for $\eps\to0$. Similarly for EOT.

For finite positive measures $\lambda,\bar\lambda$ with the same total mass, we write
\[
\Pi(\lambda,\bar\lambda)
:=
\left\{
\pi\in\mathcal M_+(\R^d\times\R^d):
(\mathrm{pr}_1)_\#\pi=\lambda,\ 
(\mathrm{pr}_2)_\#\pi=\bar\lambda
\right\},
\]
where $\mathrm{pr}_1((x,y))=x$ and  $\mathrm{pr}_2((x,y))=y$. 
We also set
\[
{\rm ROT}_{\varepsilon,p}(\lambda,\bar\lambda)
:=
\inf_{\pi\in\Pi(\lambda,\bar\lambda)}
\left\{
\int \|x-y\|^2\,d\pi
+
\varepsilon
\int
\left(
\frac{d\pi}{d(\lambda\otimes\bar\lambda)}
\right)^p
d\lambda d\bar\lambda
\right\},
\]
and
\[
{\rm EOT}_{\eps}(\lambda,\bar\lambda)
:=
\inf_{\pi\in\Pi(\lambda,\bar\lambda)}
\left\{
\int \|x-y\|^2\,d\pi(x,y)
+
\eps\,{\rm KL}(\pi\mid \lambda\otimes\bar\lambda)
\right\},
\]
both with the convention that the second term is $+\infty$ if
$\pi\not\ll \lambda\otimes\bar\lambda$.

\subsection{Auxiliary lemmas}
In this section, we prepare several auxiliary results for the proof of the upper bound.

\subsubsection{Local profiles}
We start with two calculations regarding the local profiles; they are elementary but important for the main results.

\begin{lemma}[Barenblatt profile]\label{lemma:integrals}
Let $A$ be a positive 
definite $d\times d$ matrix, $c,\eps>0$, and
$p>1$. Then
\begin{equation}
\label{eq:barenblatt}
\frac{1}{\varepsilon^{\frac{1}{p-1}}}
\int_{\mathbb R^d}
\left(c-\frac{1}{p}\|x\|_A^2\right)_+^{\frac{1}{p-1}}
dx
=
\frac{(\pi p)^{d/2}}{\sqrt{\det A}}\,
\frac{\Gamma\left(\frac{p}{p-1}\right)}
{\Gamma\left(\frac d2+\frac{p}{p-1}\right)}
\,
\frac{
c^{\frac d2+\frac{1}{p-1}}
}{
\varepsilon^{\frac{1}{p-1}}
}
\end{equation}
and
\begin{equation}
\label{eq:barenblatt-p}
\frac{1}{\varepsilon^{\frac{1}{p-1}}}
\int_{\mathbb R^d}
\left(c-\frac{1}{p}\|x\|_A^2\right)_+^{\frac{p}{p-1}}
dx
=
\frac{(\pi p)^{d/2}}{\sqrt{\det A}}\,
\frac{\Gamma\left(\frac{2p-1}{p-1}\right)}
{\Gamma\left(\frac d2+\frac{2p-1}{p-1}\right)}
\,
\frac{
c^{\frac d2+\frac{p}{p-1}}
}{
\varepsilon^{\frac{1}{p-1}}
}.
\end{equation}
As a consequence, for every $\gamma >0$, the constant
\[
c_{\varepsilon,\gamma,A}
:=
\left[
\gamma\,
\frac{\sqrt{\det A}}{\mathfrak B_{d,p}}
\right]^{
\frac{2(p-1)}{d(p-1)+2}
}
\varepsilon^{
\frac{2}{d(p-1)+2}
}
\]
solves
\[
\frac{1}{\varepsilon^{\frac{1}{p-1}}}
\int_{\mathbb R^d}
\left(c_{\varepsilon,\gamma,A}-\frac{1}{p}\|x\|_A^2\right)_+^{\frac{1}{p-1}}
dx
=
\gamma .
\]
\end{lemma}

\begin{proof}
For \(\alpha>-1\), set
\[
I_\alpha(c,A):=
\int_{\mathbb R^d}
\left(c-\frac1p\|x\|_A^2\right)_+^\alpha dx .
\]
With the change of variables \(z=A^{1/2}x\), followed by polar coordinates and
\(u=r^2/(pc)\),
\[
\begin{aligned}
I_\alpha(c,A)
&=
\frac{\omega_{d-1}}{\sqrt{\det A}}
\int_0^{\sqrt{pc}}
\left(c-\frac{r^2}{p}\right)^\alpha r^{d-1}\,dr  \\
&=
\frac{\omega_{d-1}}{2\sqrt{\det A}}\,
p^{d/2}c^{d/2+\alpha}
B\left(\frac d2,\alpha+1\right)  
=
\frac{(\pi p)^{d/2}}{\sqrt{\det A}}\,
\frac{\Gamma(\alpha+1)}
{\Gamma\left(\frac d2+\alpha+1\right)}
c^{d/2+\alpha}.
\end{aligned}
\]
Here we used
\(\omega_{d-1}=2\pi^{d/2}/\Gamma(d/2)\).
Taking \(\alpha=1/(p-1)\) and \(\alpha=p/(p-1)\), and then multiplying by
\(\varepsilon^{-1/(p-1)}\), gives
\eqref{eq:barenblatt} and \eqref{eq:barenblatt-p}.
It remains to identify $c_{\varepsilon,\gamma,A}$. Since
\[
\mathfrak B_{d,p}
=
\frac{(\pi p)^{d/2}\Gamma\left(\frac p{p-1}\right)}
{\Gamma\left(\frac d2+\frac p{p-1}\right)},
\]
the first identity gives
\[
\frac{1}{\varepsilon^{1/(p-1)}}I_{1/(p-1)}(c,A)
=
\frac{\mathfrak B_{d,p}}{\sqrt{\det A}}\,
\frac{c^{\,d/2+1/(p-1)}}{\varepsilon^{1/(p-1)}}.
\]
Solving this equation with right-hand side \(\gamma\), and using
\(
\frac d2+\frac1{p-1}
=
\frac{d(p-1)+2}{2(p-1)},
\)
yields the stated value of \(c_{\varepsilon,\gamma,A}\).
\end{proof}

\begin{lemma}[Gaussian profile]\label{lemma:eot-gaussian-profile}
Let \(A\) be  positive definite, \(a>0\), and define
\[
G_{\varepsilon,a,A}(z)
:=
\frac{\sqrt{\det A}}{a(\pi\varepsilon)^{d/2}}
\exp\left(-\frac{\|z\|_A^2}{\varepsilon}\right).
\]
Then
\[
\int_{\mathbb R^d}G_{\varepsilon,a,A}(z)\,dz
=
\frac1a, \qquad
\int_{\mathbb R^d}\|z\|_A^2G_{\varepsilon,a,A}(z)\,dz
=
\frac{d\varepsilon}{2a}
\]
and
\[
\int_{\mathbb R^d}
G_{\varepsilon,a,A}(z)\log G_{\varepsilon,a,A}(z)\,dz
=
\frac1a
\left[
\frac12\log\det A
-
\log a
-
\frac d2\log(\pi\varepsilon)
-
\frac d2
\right].
\]
Consequently,
\[
\int_{\mathbb R^d}
\left[
\|z\|_A^2G_{\varepsilon,a,A}(z)
+
\varepsilon G_{\varepsilon,a,A}(z)\log G_{\varepsilon,a,A}(z)
\right]dz
=
\frac{\varepsilon}{a}
\left[
\frac12\log\det A
-
\log a
-
\frac d2\log(\pi\varepsilon)
\right].
\]
\end{lemma}

The proof is standard and omitted.

\subsubsection{Quantization estimates}
As explained in \cref{sec:general_recipe}, the idea of our proof is to define ``leading'' subcouplings that drive the limit and complete them by ``remainder'' couplings of small mass to satisfy the marginal constraints. The following lemmas will be used to bound these remainder terms. We divide the estimates into three parts: first for \(p\in(1,2]\), then for \(p>2\), and finally for EOT.

\begin{lemma}[Quantization estimate for $p\leq2$]\label{lemma:quant-p-small} 
Fix $p\in (1,2]$ and $\lambda\in \mathcal M^+(\mathbb R^d)$ such that $\lambda \ll \cL^d$, with
\[
\int \|x\|^{2+\beta}\,d\lambda(x)<\infty
\]
for some $\beta>0$. Then there exists a constant $C_{d,p,\beta}$ such that, for every $\eta>0$ small enough,
\begin{equation}
 \label{eq:claim-quantization-finite}
{\rm ROT}_{\eta,p}(\lambda,\lambda)
\leq  
C_{d,p,\beta}
\left( \int \|x\|^{2+\beta} d\lambda +\lambda(\mathbb R^d)\right)^{1-\frac{2}{d(p-1)+2}}
\lambda(\mathbb R^d)^{\frac{2(2-p)}{d(p-1)+2}}
\eta^{\frac{2}{d(p-1)+2}} .
\end{equation}
\end{lemma}

\begin{proof}
If $\lambda(\mathbb R^d)=0$, the result is immediate by taking the zero coupling. Hence
we assume that $\lambda(\mathbb R^d)>0$. Let
\[
e_{n,2}^2(\lambda)
:=
\inf_{x_1,\dots,x_n\in\mathbb R^d}
\int \inf_{1\leq i\leq n}\|x-x_i\|^2\,d\lambda(x)
\]
be the quadratic quantization error of $\lambda$. Choose points $x_1,\dots,x_n$
whose quantization error is at most $2e_{n,2}^2(\lambda)$, and let
$(C_i)_{i=1}^n$ be the associated nearest-neighbor partition, with
ties broken such that the sets $C_i$ are Borel. Set
$m_i:=\lambda(C_i)$ and note that $\sum_i m_i=\lambda(\mathbb R^d)$. Consequently, we have
\[
\pi_n
:=
\sum_{i:m_i>0}
\frac{1}{m_i}\,
\lambda|_{C_i}\otimes \lambda|_{C_i} \in\Pi(\lambda,\lambda) \quad {\rm and} \quad 
\frac{d\pi_n}{d(\lambda\otimes \lambda)}
=
\sum_{i:m_i>0}
\frac{{\bf 1}_{C_i\times C_i}}{m_i}.
\]
Furthermore, as $t\mapsto t^{2-p}$ is concave thanks to $p\in(1,2]$,
\begin{equation}
    \label{eq:quantization-jensen}
    \int
\left(
\frac{d\pi_n}{d(\lambda\otimes \lambda)}
\right)^p
d(\lambda\otimes \lambda)
=
\sum_{i:m_i>0}m_i^{2-p}
\leq n^{p-1}\lambda(\mathbb R^d)^{2-p}.
\end{equation}
For the cost term, if $x,y\in C_i$, then
$\|x-y\|^2\leq 2\|x-x_i\|^2+2\|y-x_i\|^2$, and hence
\[
\int \|x-y\|^2\,d\pi_n(x,y)
\leq
4\sum_i\int_{C_i}\|x-x_i\|^2\,d\lambda(x)
\leq
8e_{n,2}^2(\lambda).
\]
Thus, for every $n\geq1$,
\[
{\rm ROT}_{\eta,p}(\lambda,\lambda)
\leq
8e_{n,2}^2(\lambda)+\eta n^{p-1}\lambda(\mathbb R^d)^{2-p}.
\]
By the quantization estimate \cite[Corollary~6.7]{Siegfried.Harald.Springer.2000} applied to
$\lambda/\lambda(\mathbb R^d)$,
\[
e_{n,2}^2(\lambda)
\leq
C_{d,\beta}
\left(\lambda(\mathbb R^d)+\int \|x\|^{2+\beta}d\lambda\right)n^{-2/d}.
\]
Consequently,
\[
{\rm ROT}_{\eta,p}(\lambda,\lambda)
\leq
C_{d,\beta}
\left(\lambda(\mathbb R^d)+\int \|x\|^{2+\beta}d\lambda\right)n^{-2/d}
+\eta \lambda(\mathbb R^d)^{2-p} n^{p-1},
\]
and minimizing the bound over $n\in\mathbb N$ gives \eqref{eq:claim-quantization-finite}.
\end{proof}

For $p>2$, the proof of \Cref{lemma:quant-p-small} fails at~\eqref{eq:quantization-jensen}. In this setting, we need the additional assumptions \ref{assumption:cone-2} and \ref{assumption:hoelder_densities-2}.  

\begin{lemma}[Quantization estimate for $p>2$, interior]
\label{lemma:weighted-quantization-large-p}
Fix \(p>2\). Let \(Q\subset\mathbb R^d\) be a bounded open set with Lipschitz
boundary. Let $P,\lambda \in \mathcal{M}^+(Q)$ be such that \(P=f\,d\cL^d\) and \(\lambda=\rho\,d\cL^d\). Assume that, for some \(\theta\in(0,1]\),
\[
0<\underline f\le f(x)\le \overline f<\infty,
\qquad
0<\underline a\,\theta\le \rho(x)\le \overline a\,\theta<\infty
\]
for \(\cL^d\)-a.e.~\(x\in Q\). Then there exists a constant \(C\), depending
only on \(d,p,\underline f,\overline f,\underline a,\overline a\) and on the
Lipschitz constant of \(Q\), such that, for every sufficiently small
\(\varepsilon>0\), there exists \(\pi_\varepsilon\in\Pi(\lambda,\lambda)\)
with \(\pi_\varepsilon\ll P\otimes P\) satisfying
\[
\int_{Q\times Q}\|x-y\|^2\,d\pi_\varepsilon(x,y)
+
\varepsilon
\int_{Q\times Q}
\left(
\frac{d\pi_\varepsilon}{d(P\otimes P)}
\right)^p
\,dP(x)dP(y)
\le
C |Q|\,
\theta^{1+\frac{2(p-1)}{d(p-1)+2}}
\varepsilon^{\frac{2}{d(p-1)+2}} .
\]
\end{lemma}

\begin{proof}
Since \(Q\) is bounded with Lipschitz boundary, it satisfies a uniform density
estimate: there exist constants \(c>0\) and \(r_0>0\), depending only on the
Lipschitz constant of \(Q\), such that
\[
|Q\cap B(x,r)|\ge c r^d,
\qquad x\in Q,\quad 0<r\le r_0 .
\]
Fix \(\varepsilon>0\) and set
\[
h:=\theta^{\frac{p-1}{d(p-1)+2}}
\varepsilon^{\frac{1}{d(p-1)+2}} .
\]
For \(\varepsilon>0\) sufficiently small, \(h\le r_0\). Choose a maximal
\(h\)-separated family \(x_1,\ldots,x_N\in Q\). Thus
\[
\|x_i-x_j\|\ge h,\qquad i\neq j,
\]
and, by maximality,
\[
Q\subset \bigcup_{i=1}^N B(x_i,h).
\]
Let \(C_1,\ldots,C_N\) be the associated nearest-neighbor partition of
\(Q\), with ties broken so that the partition remains measurable. Then, up to \(\cL^d\)-null sets,
\[
Q\cap B(x_i,h/2)\subset C_i\subset Q\cap B(x_i,h).
\]
Consequently,
\[
c h^d\le |C_i|\le C h^d,
\qquad i=1,\ldots,N,
\]
and, since the sets \(Q\cap B(x_i,h/2)\) are pairwise disjoint,
$N h^d\le C |Q|.$ Set $m_i:=\lambda(C_i).$ 
By the assumptions on \(\rho\), we have
\[
c\,\theta h^d\le m_i\le C\,\theta h^d,
\qquad i=1,\ldots,N.
\]
Define
\[
\pi_\varepsilon
:=
\sum_{i=1}^N
\frac{1}{m_i}\,
\lambda|_{C_i}\otimes\lambda|_{C_i}.
\]
Then \(\pi_\varepsilon\in\Pi(\lambda,\lambda)\). Indeed, for every Borel set
\(A\subset\mathbb R^d\),
\[
(\operatorname{pr}_1)_\#\pi_\varepsilon(A)
=
\sum_{i=1}^N
\lambda(A\cap C_i)
=
\lambda(A),
\]
and the same argument gives $(\operatorname{pr}_2)_\#\pi_\varepsilon=\lambda .$ We first estimate the transport part. Since \(C_i\subset B(x_i,h)\), we have
\(\operatorname{diam}(C_i)\le 2h\) and hence
\[
\int \|x-y\|^2\,d\pi_\varepsilon(x,y)
\le
C h^2\sum_{i=1}^N m_i
=
C h^2\lambda(Q)
\le
C |Q|\theta h^2 .
\]

Next, we estimate the \(L^p(P\otimes P)\)-term. On \(C_i\times C_i\),
\[
\frac{d\pi_\varepsilon}{d(P\otimes P)}(x,y)
=
\frac{\rho(x)\rho(y)}{m_i f(x)f(y)}.
\]
Therefore,
\[
\int_{C_i\times C_i}
\left(
\frac{d\pi_\varepsilon}{d(P\otimes P)}
\right)^p
dP(x)dP(y)
=
m_i^{-p}
\left(
\int_{C_i}\rho(x)^p f(x)^{1-p}\,dx
\right)^2 .
\]
Since
\[
\rho(x)^p f(x)^{1-p}
=
\rho(x)\left(\frac{\rho(x)}{f(x)}\right)^{p-1}
\le
C\theta^{p-1}\rho(x),
\]
we get
\[
\int_{C_i}\rho(x)^p f(x)^{1-p}\,dx
\le
C\theta^{p-1}m_i
\]
and hence
\[
\int_{C_i\times C_i}
\left(
\frac{d\pi_\varepsilon}{d(P\otimes P)}
\right)^p
dP(x)dP(y)
\le
C\theta^{2(p-1)}m_i^{2-p}.
\]
In view of \(p>2\) and \(m_i\ge c\theta h^d\),
\[
m_i^{2-p}\le C\theta^{2-p}h^{d(2-p)}.
\]
Summing over \(i=1,\ldots,N\) and using \(N h^d\le C|Q|\), we have
\[
\int
\left(
\frac{d\pi_\varepsilon}{d(P\otimes P)}
\right)^p
dP\,dP
\le
C\theta^{2(p-1)}\theta^{2-p}h^{d(2-p)}N
\le
C |Q|\theta^p h^{-d(p-1)} .
\]
Combining the two estimates, we obtain
\[
\int \|x-y\|^2\,d\pi_\varepsilon(x,y)
+
\varepsilon
\int
\left(
\frac{d\pi_\varepsilon}{d(P\otimes P)}
\right)^p
dP\,dP
\le
C|Q|
\left(
\theta h^2
+
\varepsilon\theta^p h^{-d(p-1)}
\right).
\]
With
\[
h=\theta^{\frac{p-1}{d(p-1)+2}}
\varepsilon^{\frac{1}{d(p-1)+2}}
\]
we have
\[
\theta h^2
=
\theta^{1+\frac{2(p-1)}{d(p-1)+2}}
\varepsilon^{\frac{2}{d(p-1)+2}}
\qquad
\text{and}
\qquad
\varepsilon\theta^p h^{-d(p-1)}
=
\theta^{1+\frac{2(p-1)}{d(p-1)+2}}
\varepsilon^{\frac{2}{d(p-1)+2}}.
\]
Therefore,
\[
\int \|x-y\|^2\,d\pi_\varepsilon(x,y)
+
\varepsilon
\int
\left(
\frac{d\pi_\varepsilon}{d(P\otimes P)}
\right)^p
dP\,dP
\le
C |Q|\,
\theta^{1+\frac{2(p-1)}{d(p-1)+2}}
\varepsilon^{\frac{2}{d(p-1)+2}},
\]
which proves the claim.
\end{proof}

While the preceding lemma will be used for the bulk of the remainder coupling on the interior, the next estimate treats the complementary boundary layer. It uses the same local construction, but is stated separately because the set  \(G_\delta\) in its statement need not have a Lipschitz scale uniform in \(\delta\) or in
the relative position of the cubes. The point is that the constant in its assertion is uniform w.r.t.\ this geometry. In the main proof of the upper bound, the boundary contribution will be killed by \(|G_\delta|\to0\), while the bulk contribution will be killed by \(\theta\downarrow0\).

\begin{lemma}[Quantization estimate for $p>2$, boundary layer]
\label{le:boundary-layer-large-p}
Fix $p>2.$ 
Let $Q\subset\mathbb R^d$ be a bounded open set with Lipschitz boundary, and let
$I_1,\ldots,I_N\Subset Q$ be pairwise disjoint axis-parallel cubes. For
$\delta\ge0$, set
\[
G_\delta:=Q\setminus \bigcup_{i=1}^N I_i^\delta,
\qquad
I_i^\delta:=\{x\in I_i:\operatorname{dist}(x,\partial I_i)>\delta\}.
\]
Let $P,\lambda \in \mathcal{M}^+(G_\delta)$ be such that $P=f\,d\cL^d$ and $\lambda=\rho\,d\cL^d$. Assume that, $\cL^d$-a.e.\ on $G_\delta$,
\[
0<\underline f\le f\le \overline f<\infty,
\qquad
0<\underline a\le \rho\le \overline a<\infty .
\]
Then there exists \(\varepsilon_0>0\) such that for every $0<\varepsilon\leq\varepsilon_0$, there exists
$\pi_\varepsilon\in\Pi(\lambda,\lambda)$ with $\pi_\varepsilon\ll P\otimes P$ and
\[
\int\|x-y\|^2\,d\pi_\varepsilon(x,y)
+
\varepsilon\int
\left(\frac{d\pi_\varepsilon}{d(P\otimes P)}\right)^p dP\,dP
\le
C |G_\delta|\,\varepsilon^{\frac{2}{d(p-1)+2}},
\]
where $C$ depends only on $d,p,\underline f,\overline f,\underline a,\overline a$ and on the
Lipschitz constant of $Q$. In particular, $C$ is independent of $N$ and $\delta$; only $\varepsilon_0$ may depend on the family of cubes and on~$\delta$.
\end{lemma}

\begin{proof}
If \(|G_\delta|=0\), the claim follows by taking the zero coupling. Since
\(P,\lambda\ll\cL^d\), we may modify \(G_\delta\) on a null set.

Let
\[
K_\delta:=\bigcup_{i:I_i^\delta\neq\varnothing}\overline{I_i^\delta}.
\]
If \(K_\delta=\varnothing\), set \(H_\delta:=Q\). Otherwise, let
\(H_\delta\) be obtained from \(Q\setminus K_\delta\) by removing the finitely
many coordinate hyperplanes containing a face of one of the boxes
\(\overline{I_i^\delta}\). Then
\[
H_\delta\subset G_\delta,
\qquad
|G_\delta\setminus H_\delta|=0,
\]
so \(P\) and \(\lambda\) may be regarded as measures on \(H_\delta\).

We claim that there are \(c_0>0\), depending only on \(d\) and on the
Lipschitz constant of \(Q\), and \(r_0>0\), possibly depending on the cube
configuration and on \(\delta\), such that
\begin{equation}
    \label{eq:boundary-layer-density-large-p}
    |H_\delta\cap \ball(x,r)|\ge c_0 r^d,
    \qquad x\in H_\delta,\quad 0<r\le r_0 .
\end{equation}
For \(K_\delta=\varnothing\), this is the standard lower density estimate for Lipschitz domains. Otherwise, since \(K_\delta\Subset Q\), choose \(r_0\)
smaller than the localization radius for \(Q\), smaller than a fixed fraction
of \(\operatorname{dist}(K_\delta,\partial Q)\), and smaller than one half of
the minimal positive side-length of the rectangular cells in the coordinate
grid generated by the faces of the boxes \(\overline{I_i^\delta}\). If
\(\ball(x,r)\cap K_\delta=\varnothing\), the Lipschitz lower density estimate
for \(Q\) gives \eqref{eq:boundary-layer-density-large-p}. If
\(\ball(x,r)\cap K_\delta\neq\varnothing\), then
\(\ball(x,r)\subset Q\), and \(x\) belongs to a rectangular grid cell contained
in \(\mathbb R^d\setminus K_\delta\). By the choice of \(r_0\), one coordinate
orthant of \(\ball(x,r)\) with vertex \(x\) is contained in that cell. This
again gives \eqref{eq:boundary-layer-density-large-p}, after decreasing
\(c_0\). Thus only the admissible scale \(r_0\), not the density constant,
depends on \(N\) or \(\delta\).

Set $h:=\varepsilon^{1/(d(p-1)+2)}$, and take \(\varepsilon\) small enough that \(h\le r_0\). Let
\(x_1,\ldots,x_M\) be a maximal \(h\)-separated family in \(H_\delta\), and let
\(C_1,\ldots,C_M\) be a measurable nearest-neighbor partition of \(H_\delta\).
Then, up to null sets,
\[
H_\delta\cap \ball(x_k,h/3)\subset C_k\subset H_\delta\cap \ball(x_k,h),
\]
and hence
\begin{equation}
    \label{eq:boundary-layer-cell-bounds}
    c h^d\le |C_k|\le C h^d,
    \qquad
    M h^d\le C |G_\delta|.
\end{equation}
The constants here are independent of \(N\) and \(\delta\).

Set \(m_k:=\lambda(C_k)\), and define
\[
\pi_\varepsilon
:=
\sum_{k=1}^M
\frac{1}{m_k}\,
\lambda|_{C_k}\otimes\lambda|_{C_k}.
\]
Then \(\pi_\varepsilon\in\Pi(\lambda,\lambda)\) and
\(\pi_\varepsilon\ll P\otimes P\). Since \(f\) and \(\rho\) are bounded above
and below on \(G_\delta\), the estimates from the proof of
\Cref{lemma:weighted-quantization-large-p}, now with \(\theta=1\) and
\(H_\delta\) in place of \(Q\), give
\[
\int\|x-y\|^2\,d\pi_\varepsilon
\le
C |G_\delta| h^2
\]
and
\[
\int
\left(
\frac{d\pi_\varepsilon}{d(P\otimes P)}
\right)^p
dP\,dP
\le
C |G_\delta| h^{-d(p-1)} .
\]
Therefore
\[
\int\|x-y\|^2\,d\pi_\varepsilon
+
\varepsilon
\int
\left(
\frac{d\pi_\varepsilon}{d(P\otimes P)}
\right)^p
dP\,dP
\le
C |G_\delta|
\left(
h^2+\varepsilon h^{-d(p-1)}
\right).
\]
As \(h=\varepsilon^{1/(d(p-1)+2)}\), this is the desired bound.
\end{proof}

Finally, we give the quantization estimate for EOT. 
\begin{lemma}[Quantization estimate for EOT]\label{lemma:quant-eot}
Let \(\lambda\) be a finite positive measure on \(\mathbb R^d\) with finite \(2+\beta\) moment for some \(\beta>0\). Then there exists a constant
\(C_{d,\beta}>0\) such that, for every \(\eta\in(0,e^{-1})\),
\[
{\rm EOT}_{\eta}(\lambda,\lambda)
\leq
-\frac d2\,\lambda(\mathbb R^d) \,\eta\log(\eta)
+
C_{d,\beta}\eta
\left[
\lambda(\mathbb R^d)+\int_{\mathbb R^d}\|x\|^{2+\beta}\,d\lambda(x)+\lambda(\mathbb R^d) |\log \lambda(\mathbb R^d) |
\right].
\]
\end{lemma}

\begin{proof}
If \(\lambda(\R^d)=0\), the result is immediate by taking the zero coupling. Hence assume \(\lambda(\R^d)>0\), and set $P:=\frac{\lambda}{\lambda(\R^d)}.$  Define 
\(x_1,\ldots,x_N\) and  
\[
(C_i)_{i=1}^N,
\qquad 
p_i:=P(C_i),
\qquad
m_i:=\lambda(C_i)=\lambda(\R^d) p_i
\]
as in \Cref{lemma:quant-p-small} for $P$. 
Set
\[
\pi_N
:=
\sum_{i:m_i>0}
\frac{1}{m_i}\,
\lambda|_{C_i}\otimes\lambda|_{C_i} \in \Pi(\lambda,\lambda), 
\]
which satisfies
\[
\frac{d\pi_N}{d(\lambda\otimes\lambda)}
=
\sum_{i:m_i>0}
\frac{\mathbf 1_{C_i\times C_i}}{m_i}.
\]
For the transport cost, we derive $\int \|x-y\|^2\,d\pi_N(x,y)
\leq 
8 \lambda(\R^d) e_{N,2}^2(P)$ as in \cref{lemma:quant-p-small}.

Using the quantization estimate $e_{n,2}^2(P)
\leq
C_{d,\beta}
\left(1+\int \|x\|^{2+\beta}dP\right)N^{-2/d}$ and the identity
\[
\lambda(\R^d) \int \|x\|^{2+\beta}\,dP(x)
=
\int \|x\|^{2+\beta}\,d\lambda(x),
\]
we then get
\begin{equation}
    \label{eq:EOT-quant-cost}
    \int \|x-y\|^2\,d\pi_N(x,y)
\leq
C_{d,\beta}
\left[
\lambda(\mathbb R^d)+\int_{\mathbb R^d}\|x\|^{2+\beta}\,d\lambda(x)
\right]
N^{-2/d}.
\end{equation}

For the entropy term,
\[
\begin{aligned}
{\rm KL}(\pi_N\mid \lambda\otimes\lambda)
&=
\sum_{i:m_i>0}
m_i\log\left(\frac{1}{m_i}\right)  \\
&=
\lambda(\mathbb R^d)\sum_{i:p_i>0}p_i\log\left(\frac{1}{p_i}\right)
+
\lambda(\mathbb R^d)\log\left(\frac{1}{\lambda(\mathbb R^d)}\right).
\end{aligned}
\]
Since $(0,\infty)\ni t\mapsto \log(t)$ is concave, Jensen's inequality yields
\[
\sum_{i:p_i>0}p_i\log\left(\frac1{p_i}\right)
\leq
\log\left(\sum_{i:p_i>0}p_i\frac1{p_i}\right)
\leq \log N,
\]
so that
\[
{\rm KL}(\pi_N\mid \lambda\otimes\lambda)
\leq
\lambda(\mathbb R^d) \log N + \lambda(\mathbb R^d)|\log  \lambda(\mathbb R^d)|.
\]
Now choose
\(N_\eta:=\left\lceil \eta^{-d/2}\right\rceil\).
Since \(\eta\in(0,e^{-1})\),
\[
N_\eta^{-2/d}\leq \eta,
\qquad
\log N_\eta
\leq
-\frac d2\log\eta+\log 2,
\]
and we conclude using \eqref{eq:EOT-quant-cost}. 
\end{proof}

For ease of reference, we record the following corollary, obtained by applying \Cref{lemma:quant-eot} with parameter \(\eta/L\) and multiplying the resulting bound by \(L\).

\begin{corollary}\label{cor:eot-quant-scaled}
Let \(L>0\). Under the assumptions of \Cref{lemma:quant-eot}, there exists
\(C_{d,\beta,L}>0\) such that, for every \(\eta\in(0,Le^{-1})\),
\begin{multline*}
    \inf_{\pi\in\Pi(\lambda,\lambda)}
\left\{
L\int \|x-y\|^2\,d\pi(x,y)
+
\eta\,{\rm KL}(\pi\mid \lambda\otimes\lambda)
\right\}
\\
\leq
-\frac d2\,\lambda(\mathbb R^d) \,\eta\log(\eta)
+
C_{d,\beta,L}\eta
\left[
\lambda(\mathbb R^d)+\int_{\mathbb R^d}\|x\|^{2+\beta}\,d\lambda(x)+\lambda(\mathbb R^d) |\log \lambda(\mathbb R^d) |
\right].
\end{multline*}
\end{corollary}

\proofreadhere

\subsection{Proof of \Cref{theorem:main-p} for $p\in(1,2]$}
Throughout this subsection, fix $p\in(1,2]$ and set
\[
D:=d(p-1)+2,
\qquad
\alpha:=\frac{p-1}{D},
\qquad
s:=1-2\alpha=\frac{d(p-1)-2p+4}{D}.
\]
Notice that $s\ge0$ and, by \Cref{rk:finitenessOfLimit},
$f_\mu^s\in L^1(\mathbb R^d)$. For $n\in\mathbb N$, let
\[
J_n:=\{0,1,\dots,2n\,2^n-1\}^d
\]
and, for $j=(j_1,\dots,j_d)\in J_n$, set
\[
{\mathfrak Q}_{j,n}
:=
\prod_{\ell=1}^d
\left[
-n+\frac{j_\ell}{2^n},
\,-n+\frac{j_\ell+1}{2^n}
\right).
\]
We regard the family \(({\mathfrak Q}_{j,n})_{j\in J_n}\) as an a.e.\  
partition of \([-n,n]^d\). For
$0\le k\le n2^n-1$, define
\[
I_{k,n}:=
\left\{x:\frac{k}{2^n}\le f_\mu(x)<\frac{k+1}{2^n}\right\},
\qquad
I_{n2^n,n}:=\{x:f_\mu(x)\ge n\},
\]
and set $I_{j,k,n}:={\mathfrak Q}_{j,n}\cap I_{k,n}$. Finally, put
\[
f_n(x):=
\sum_{j\in J_n}\sum_{k=0}^{n2^n}
\frac{k}{2^n}\mathbf 1_{I_{j,k,n}}(x).
\]
Then $0\le f_n\le f_\mu$ and $f_n\uparrow f_\mu$ almost everywhere.

For $k\ge1$ and $I_{j,k,n}\ne\varnothing$, choose
$x_{j,k,n}\in I_{j,k,n}$ such that
\begin{equation}
    \label{eq:defx_j-k-n}
    \det\nabla^2\varphi(x_{j,k,n})
    \le
    \inf_{x\in I_{j,k,n}}\det\nabla^2\varphi(x)+\frac1n,
\end{equation}
and set $A_{j,k,n}:=\nabla^2\varphi(x_{j,k,n})$. If
$I_{j,k,n}=\varnothing$, choose $A_{j,k,n}$ arbitrarily. Let
\[
c_{j,k,n}^\varepsilon
:=
c_{\varepsilon,2^n/k,A_{j,k,n}}
\]
be the constant from \Cref{lemma:integrals}, and define
\[
\xi_{j,k,n}^\varepsilon(z)
:=
\varepsilon^{-\frac1{p-1}}
\left(c_{j,k,n}^\varepsilon-\frac1p\|z\|_{A_{j,k,n}}^2\right)_+^{\frac1{p-1}}.
\]
For $k\ge1$, define
\[
\pi_{j,k,n}^\varepsilon
:=
\frac{k^2}{4^n}
\xi_{j,k,n}^\varepsilon(x-x')
\mathbf 1_{I_{j,k,n}}(x)\mathbf 1_{I_{j,k,n}}(x')\,dx\,dx',
\]
and set $\pi_{j,0,n}^\varepsilon:=0$. Let
\[
\pi_{n,{\rm lead}}^\varepsilon
:=
\sum_{j\in J_n}\sum_{k=1}^{n2^n}\pi_{j,k,n}^\varepsilon.
\]
Since
\[
\int_{\mathbb R^d}\xi_{j,k,n}^\varepsilon(z)\,dz=\frac{2^n}{k},
\]
the first and second marginals of $\pi_{n,{\rm lead}}^\varepsilon$ have the
same density, denoted by $q_n^\varepsilon$, and
\[
0\le q_n^\varepsilon\le f_n\le f_\mu.
\]
Thus
\[
\lambda_{n,{\rm rem}}^\varepsilon
:=
\mu-(\operatorname{pr}_1)_\#\pi_{n,{\rm lead}}^\varepsilon
=
\mu-(\operatorname{pr}_2)_\#\pi_{n,{\rm lead}}^\varepsilon
=(f_\mu-q_n^\varepsilon)\,d\cL^d
\]
is a finite positive measure. Let
$\pi_{n,{\rm rem}}^\varepsilon\in
\Pi(\lambda_{n,{\rm rem}}^\varepsilon,
\lambda_{n,{\rm rem}}^\varepsilon)$, to be chosen below with
$\pi_{n,{\rm rem}}^\varepsilon\ll
\lambda_{n,{\rm rem}}^\varepsilon\otimes
\lambda_{n,{\rm rem}}^\varepsilon$, and set
\begin{equation}
    \label{eq:construction-p-small}
    \pi_n^\varepsilon
    :=\pi_{n,{\rm lead}}^\varepsilon+
      \pi_{n,{\rm rem}}^\varepsilon,
    \qquad
    \gamma_n^\varepsilon
    :=(\operatorname{Id}\times\nabla\varphi)_\#\pi_n^\varepsilon .
\end{equation}
Then $\pi_n^\varepsilon\in\Pi(\mu,\mu)$ and
$\gamma_n^\varepsilon\in\Pi(\mu,\nu)$.

By \Cref{lemma:MA-change-of-variables}, applied to
\(\lambda=\pi_n^\varepsilon\), the \(L^p\)-penalty is invariant under
\((x,x')\mapsto(x,\nabla\varphi(x'))\). Hence
\[
\left\|\frac{d\gamma_n^\varepsilon}{d(\mu\otimes\nu)}\right\|_{L^p(\mu\otimes\nu)}^p
=
\left\|\frac{d\pi_n^\varepsilon}{d(\mu\otimes\mu)}\right\|_{L^p(\mu\otimes\mu)}^p.
\]
Fix $\delta'>0$. The supports \(I_{j,k,n}\times I_{j,k,n}\) of the leading blocks are pairwise
disjoint a.e., so their \(L^p\)-contributions add. (The remainder block need not be disjoint from them.) Hence, using
$(a+b)^p\le(1+\delta')a^p+C_{p,\delta'}b^p$,
\begin{equation}
    \label{young-ineq}
\left\|\frac{d\pi_n^\varepsilon}{d(\mu\otimes\mu)}\right\|_{L^p(\mu\otimes\mu)}^p
\le
(1+\delta')
\sum_{j\in J_n}\sum_{k=1}^{n2^n}
\left\|\frac{d\pi_{j,k,n}^\varepsilon}{d(\mu\otimes\mu)}\right\|_{L^p(\mu\otimes\mu)}^p
+
C_{p,\delta'}
\left\|\frac{d\pi_{n,{\rm rem}}^\varepsilon}{d(\mu\otimes\mu)}\right\|_{L^p(\mu\otimes\mu)}^p .
\end{equation}
Moreover, since $\lambda_{n,{\rm rem}}^\varepsilon\le\mu$,
\[
\left\|\frac{d\pi_{n,{\rm rem}}^\varepsilon}{d(\mu\otimes\mu)}\right\|_{L^p(\mu\otimes\mu)}^p
\le
\left\|\frac{d\pi_{n,{\rm rem}}^\varepsilon}
{d(\lambda_{n,{\rm rem}}^\varepsilon\otimes
\lambda_{n,{\rm rem}}^\varepsilon)}\right\|_{L^p(\lambda_{n,{\rm rem}}^\varepsilon\otimes
\lambda_{n,{\rm rem}}^\varepsilon)}^p.
\]
By \Cref{lemma:integrated-bregman-identity},
\[
\int\|x-y\|^2\,d\gamma_n^\varepsilon(x,y)-{\rm OT}(\mu,\nu)
=
2\int \mathbb D(x,\nabla\varphi(x'))\,d\pi_n^\varepsilon(x,x').
\]
Accordingly, define the leading contribution
\[
\mathcal T_\varepsilon^{\delta'}
:=
\sum_{j\in J_n}\sum_{k=1}^{n2^n}
\left(
2\int\mathbb D(x,\nabla\varphi(x'))\,d\pi_{j,k,n}^\varepsilon
+
(1+\delta')\varepsilon
\left\|\frac{d\pi_{j,k,n}^\varepsilon}{d(\mu\otimes\mu)}\right\|_{L^p(\mu\otimes\mu)}^p
\right)
\]
and the scaled remainder
\[
\mathcal R_{\varepsilon,n}
:=
\varepsilon^{-2/D}
\left(
2\int\mathbb D(x,\nabla\varphi(x'))\,d\pi_{n,{\rm rem}}^\varepsilon
+
C_{p,\delta'}\varepsilon
\left\|\frac{d\pi_{n,{\rm rem}}^\varepsilon}
{d(\lambda_{n,{\rm rem}}^\varepsilon\otimes
\lambda_{n,{\rm rem}}^\varepsilon)}\right\|_{L^p(\lambda_{n,{\rm rem}}^\varepsilon\otimes
\lambda_{n,{\rm rem}}^\varepsilon)}^p
\right).
\]
The estimates above imply, for every fixed $n$,
\begin{equation}
    \label{eq:upper-split-small-p}
\frac{{\rm ROT}_{\varepsilon,p}(\mu,\nu)-{\rm OT}(\mu,\nu)}
{\varepsilon^{2/D}}
\le
\varepsilon^{-2/D}\mathcal T_\varepsilon^{\delta'}
+
\mathcal R_{\varepsilon,n}.
\end{equation}

\begin{lemma}\label{lemma:control-leading}
Under the assumptions of \Cref{theorem:main-p}(i), for every $\delta'>0$,
\[
\limsup_{n\uparrow\infty}\limsup_{\varepsilon\downarrow0}
\varepsilon^{-2/D}\mathcal T_\varepsilon^{\delta'}
\le
(1+\delta')\mathfrak C_{d,p}\mathcal K(\mu,\nu,p).
\]
\end{lemma}

\begin{proof}
Fix $R>0$, and let $J_{n,R}$ be the set of $j\in J_n$ such that
${\mathfrak Q}_{j,n}\subset\ball(0,R)$. By \Cref{lemma:Bregman_ptwisebd}, the
uniform ellipticity, and the uniform continuity of $\nabla^2\varphi$ on
$\overline{\ball(0,R)}$, there is $\omega_{n,R}\downarrow0$ (with $R$ fixed), such that
\[
2\mathbb D(x,\nabla\varphi(x'))
\le
(1+\omega_{n,R})\|x-x'\|_{A_{j,k,n}}^2,
\qquad
x,x'\in I_{j,k,n},\quad j\in J_{n,R}.
\]
For $k\ge1$ and $x,x'\in I_{j,k,n}$,
\[
\frac{d\pi_{j,k,n}^\varepsilon}{d(\mu\otimes\mu)}(x,x')
=
\frac{(k^2/4^n)\xi_{j,k,n}^\varepsilon(x-x')}{f_\mu(x)f_\mu(x')},
\]
and $f_\mu\ge k2^{-n}$ on $I_{j,k,n}$. Therefore
\[
\left\|\frac{d\pi_{j,k,n}^\varepsilon}{d(\mu\otimes\mu)}\right\|_{L^p(\mu\otimes\mu)}^p
\le
\frac{k^2}{4^n}
\int_{I_{j,k,n}}\int_{\mathbb R^d}
\left(\xi_{j,k,n}^\varepsilon(x-x')\right)^p\,dx'\,dx .
\]
It follows that, for $j\in J_{n,R}$,
\[
\begin{aligned}
\mathcal T_{j,k,n}^\varepsilon
&:=
2\int\mathbb D(x,\nabla\varphi(x'))\,d\pi_{j,k,n}^\varepsilon
+
(1+\delta')\varepsilon
\left\|\frac{d\pi_{j,k,n}^\varepsilon}{d(\mu\otimes\mu)}\right\|_{L^p(\mu\otimes\mu)}^p
\\
&\le
(1+\omega_{n,R}+\delta')|I_{j,k,n}|\frac{k^2}{4^n}
\frac1{\varepsilon^{1/(p-1)}}
\int_{\mathbb R^d}
\left[
\|z\|_{A_{j,k,n}}^2
\left(c_{j,k,n}^\varepsilon-\frac1p\|z\|_{A_{j,k,n}}^2\right)_+^{\frac1{p-1}}
\right.
\\[-.35em]
&\hspace{20em}\left.
+
\left(c_{j,k,n}^\varepsilon-\frac1p\|z\|_{A_{j,k,n}}^2\right)_+^{\frac p{p-1}}
\right]dz .
\end{aligned}
\]
The same beta-function computation as in \Cref{lemma:integrals} gives
\[
\begin{aligned}
&\frac1{\varepsilon^{1/(p-1)}}
\int_{\mathbb R^d}
\left[
\|z\|_{A_{j,k,n}}^2
\left(c_{j,k,n}^\varepsilon-\frac1p\|z\|_{A_{j,k,n}}^2\right)_+^{\frac1{p-1}}
+
\left(c_{j,k,n}^\varepsilon-\frac1p\|z\|_{A_{j,k,n}}^2\right)_+^{\frac p{p-1}}
\right]dz
\\
&\hspace{8em}=
\mathfrak C_{d,p}
\left(\frac{2^n}{k}\right)^{\frac{d(p-1)+2p}{D}}
\varepsilon^{2/D}
(\det A_{j,k,n})^\alpha .
\end{aligned}
\]
Consequently, for $j\in J_{n,R}$,
\begin{equation}
    \label{eq:leading-cell-small-p}
\mathcal T_{j,k,n}^\varepsilon
\le
(1+\omega_{n,R}+\delta')\mathfrak C_{d,p}
|I_{j,k,n}|
\left(\frac{2^n}{k}\right)^{-s}
\varepsilon^{2/D}
(\det A_{j,k,n})^\alpha .
\end{equation}
For $j\notin J_{n,R}$, the same estimate, using the global quadratic bound in
\Cref{lemma:Bregman_ptwisebd} and the ellipticity of $A_{j,k,n}$, gives
\begin{equation}
    \label{eq:leading-tail-cell-small-p}
\mathcal T_{j,k,n}^\varepsilon
\le
C|I_{j,k,n}|
\left(\frac{2^n}{k}\right)^{-s}
\varepsilon^{2/D}.
\end{equation}

By \eqref{eq:MA-density-identity}, together with
\eqref{eq:defx_j-k-n}, for a.e. $x\in I_{j,k,n}$,
\[
\det A_{j,k,n}
\le
\frac{f_\mu(x)}{f_\nu(\nabla\varphi(x))}+\frac1n.
\]
Since $(2^n/k)^{-s}=(k2^{-n})^s\le f_\mu^s$ on $I_{j,k,n}$,
\[
\begin{aligned}
S_n
&:=
\sum_{j\in J_n}\sum_{k=1}^{n2^n}
|I_{j,k,n}|
\left(\frac{2^n}{k}\right)^{-s}
(\det A_{j,k,n})^\alpha
\\
&\le
\int_{\{f_\mu>0\}}
 f_\mu^s
\left(\frac{f_\mu}{f_\nu(\nabla\varphi)}+\frac1n\right)^\alpha dx
\\
&\le
\mathcal K(\mu,\nu,p)+n^{-\alpha}\int_{\mathbb R^d}f_\mu^s\,dx,
\end{aligned}
\]
where we used $(a+b)^\alpha\le a^\alpha+b^\alpha$ and the definition of
$\mathcal K(\mu,\nu,p)$. Moreover, if $j\notin J_{n,R}$, then
$I_{j,k,n}\subset\mathbb R^d\setminus\ball(0,R-\sqrt d\,2^{-n})$, and therefore
\[
\sum_{j\notin J_{n,R}}\sum_{k=1}^{n2^n}
|I_{j,k,n}|
\left(\frac{2^n}{k}\right)^{-s}
\le
\int_{\mathbb R^d\setminus\ball(0,R-\sqrt d\,2^{-n})}f_\mu^s\,dx .
\]
Combining the last two estimates with
\eqref{eq:leading-cell-small-p} and \eqref{eq:leading-tail-cell-small-p}, then
letting first $n\to\infty$ and afterwards $R\to\infty$, gives the claim.
\end{proof}

\begin{lemma}\label{lemma:first-marginal-convergence}
For every fixed $n$,
\[
q_n^\varepsilon
=
\frac{d(\operatorname{pr}_1)_\#\pi_{n,{\rm lead}}^\varepsilon}{d\cL^d}
=
\frac{d(\operatorname{pr}_2)_\#\pi_{n,{\rm lead}}^\varepsilon}{d\cL^d}
\longrightarrow f_n
\qquad \cL^d\text{-a.e.\ as }\varepsilon\downarrow0.
\]
Consequently, if
\[
m_n^\varepsilon:=\lambda_{n,{\rm rem}}^\varepsilon(\mathbb R^d),
\qquad
M_n^\varepsilon:=\int\|x\|^{2+\beta}\,d\lambda_{n,{\rm rem}}^\varepsilon(x),
\]
then
\[
m_n^\varepsilon\to m_n:=\int(f_\mu-f_n)\,dx,
\qquad
M_n^\varepsilon\to M_n:=\int\|x\|^{2+\beta}(f_\mu-f_n)\,dx .
\]
\end{lemma}

\begin{proof}
Fix $n,j,k$. The case $k=0$ or $I_{j,k,n}=\varnothing$ is trivial, so assume
$k\ge1$ and set $I:=I_{j,k,n}$, $A:=A_{j,k,n}$, and
$\gamma:=2^n/k$. With
\[
\eta_\varepsilon(z):=\gamma^{-1}\xi_{j,k,n}^\varepsilon(z),
\]
\Cref{lemma:integrals} gives $\int\eta_\varepsilon=1$. Moreover
$c_{j,k,n}^\varepsilon=O(\varepsilon^{2/D})$, and the ellipticity of $A$ gives
\[
\operatorname{supp}\eta_\varepsilon\subset\ball(0,C\varepsilon^{1/D}),
\qquad
0\le\eta_\varepsilon\le C\varepsilon^{-d/D}
\mathbf 1_{\ball(0,C\varepsilon^{1/D})}.
\]
Thus $(\eta_\varepsilon)$ is an approximate identity. By the Lebesgue
differentiation theorem,
\[
\int\eta_\varepsilon(z)\mathbf 1_I(x-z)\,dz
\to
\mathbf 1_I(x)
\qquad\text{for }\cL^d\text{-a.e.\ }x.
\]
The first marginal density of $\pi_{j,k,n}^\varepsilon$ is
\[
\frac{k^2}{4^n}\mathbf 1_I(x)
\int_I\xi_{j,k,n}^\varepsilon(x-x')\,dx'
=
\frac{k}{2^n}\mathbf 1_I(x)
\int\eta_\varepsilon(z)\mathbf 1_I(x-z)\,dz,
\]
and hence it converges a.e.\ to $(k/2^n)\mathbf 1_I$. The second marginal is
the same by symmetry. Summing over the finitely many pairs $(j,k)$ proves
$q_n^\varepsilon\to f_n$ a.e. Since $0\le q_n^\varepsilon\le f_n\le f_\mu$,
dominated convergence gives the two asserted limits for the residual mass and
moment.
\end{proof}

\begin{proof}[Proof of \Cref{theorem:main-p} for {$p\in (1,2]$}]
It remains to choose the remainder coupling and pass to the limits. Let
$L<\infty$ be such that
\[
2\mathbb D(x,\nabla\varphi(x'))\le L\|x-x'\|^2,
\qquad x,x'\in\mathbb R^d,
\]
which exists by \Cref{lemma:Bregman_ptwisebd}. If
$m_n^\varepsilon=0$, take $\pi_{n,{\rm rem}}^\varepsilon=0$. Otherwise, apply
the construction from the proof of \Cref{lemma:quant-p-small} to
$\lambda_{n,{\rm rem}}^\varepsilon$, with regularization parameter comparable
to $\varepsilon$, and optimize over the number of cells. This gives
\begin{equation}
    \label{eq:bound-remainders-small-p-1}
\mathcal R_{\varepsilon,n}
\le
C_{\delta'}
\left[
(M_n^\varepsilon+m_n^\varepsilon)^{1-2/D}
(m_n^\varepsilon)^{\frac{2(2-p)}{D}}
+
\varepsilon^{1-2/D}(m_n^\varepsilon)^{2-p}
\right],
\end{equation}
with the right-hand side read as zero when $m_n^\varepsilon=0$.

By \Cref{lemma:first-marginal-convergence}, for fixed $n$,
$m_n^\varepsilon\to m_n$ and $M_n^\varepsilon\to M_n$. Since
$0\le m_n^\varepsilon\le1$ and $1-2/D>0$, \eqref{eq:bound-remainders-small-p-1}
yields
\[
\limsup_{\varepsilon\downarrow0}\mathcal R_{\varepsilon,n}
\le
C_{\delta'}(M_n+m_n)^{1-2/D}.
\]
As $f_n\uparrow f_\mu$ a.e.\ and $\mu\in\mathcal P_{2+\beta}(\mathbb R^d)$,
we have $m_n\to0$ and $M_n\to0$. Therefore
\begin{equation}
    \label{eq:bound-remainders-small-p-2}
\limsup_{n\uparrow\infty}\limsup_{\varepsilon\downarrow0}
\mathcal R_{\varepsilon,n}=0.
\end{equation}
Combining \eqref{eq:upper-split-small-p}, \Cref{lemma:control-leading}, and
\eqref{eq:bound-remainders-small-p-2}, we obtain
\[
\limsup_{\varepsilon\downarrow0}
\frac{{\rm ROT}_{\varepsilon,p}(\mu,\nu)-{\rm OT}(\mu,\nu)}
{\varepsilon^{2/D}}
\le
(1+\delta')\mathfrak C_{d,p}\mathcal K(\mu,\nu,p).
\]
Letting $\delta'\downarrow0$ completes the proof of the upper bound for
$p\in(1,2]$, and hence, together with the lower bound, the proof of
\Cref{theorem:main-p} in this case.
\end{proof}

\subsection{Proof of \Cref{theorem:main-p} for \(p>2\)}

\begin{proof}[Proof of \Cref{theorem:main-p} for \(p>2\)]
It remains to prove the upper bound. Assume
\cref{assumption:hoelder_densities}, \cref{assumption:elliptic_phi} and \cref{assumption:cone-2}, \cref{assumption:hoelder_densities-2}. Set
\[
D:=d(p-1)+2,
\qquad
\alpha:=\frac{p-1}{D},
\qquad
\kappa:=1+2\alpha=\frac{d(p-1)+2p}{D}.
\]
Under the present assumptions there are constants
\[
0<\underline f_\mu\le f_\mu(x)\le \overline f_\mu<\infty,
\qquad x\in \operatorname{supp}\mu .
\]
Moreover, by \eqref{eq:MA-density-identity} and since
\(|\partial\Omega_0|=0\),
\begin{equation}
    \label{eq:K-large-p-det-form}
\mathcal K(\mu,\nu,p)
=
\int_{\Omega_0}
f_\mu(x)^{1-2\alpha}
\bigl(\det\nabla^2\varphi(x)\bigr)^\alpha\,dx .
\end{equation}

Fix \(\theta\in(0,1)\) and \(\delta'>0\). Let
\(\mathcal P=\{I_1,\ldots,I_N\}\) be a finite family of pairwise disjoint
axis-parallel cubes compactly contained in \(\Omega_0\). We shall later take
families with \(\operatorname{mesh}(\mathcal P):=\max_i\operatorname{diam}(I_i)\)
tending to zero and with \(\bigcup_i I_i\) exhausting \(\Omega_0\). Since
\(f_\mu\) is uniformly continuous and bounded away from zero, we may restrict to families fine enough that
\begin{equation}
    \label{eq:oscillation-large-p}
\sup_{x,y\in I_i}|f_\mu(x)-f_\mu(y)|
\le \theta f_i,
\qquad
f_i:=\inf_{I_i}f_\mu,
\qquad i=1,\ldots,N .
\end{equation}
Choose \(x_i\in I_i\), and set \(A_i:=\nabla^2\varphi(x_i)\). For
\(0<\delta<\frac12\min_i\operatorname{side}(I_i)\), define
\[
I_i^\delta
:=
\{x\in I_i:\operatorname{dist}(x,\partial I_i)>\delta\}.
\]
Let
\[
c_i^{\varepsilon,\theta}
:=
c_{\varepsilon,(1-\theta)f_i^{-1},A_i},
\]
where \(c_{\varepsilon,\gamma,A}\) is the constant from
\Cref{lemma:integrals}, and define
\[
\xi_i^{\varepsilon,\theta}(z)
:=
\varepsilon^{-\frac1{p-1}}
\left(
c_i^{\varepsilon,\theta}
-
\frac1p\|z\|_{A_i}^2
\right)_+^{\frac1{p-1}} .
\]
Thus
\[
\int_{\mathbb R^d}\xi_i^{\varepsilon,\theta}(z)\,dz
=
(1-\theta)f_i^{-1}.
\]
For \(\varepsilon\) sufficiently small (depending on
\(\mathcal P,\delta,\theta\)), the support of
\(\xi_i^{\varepsilon,\theta}\) is contained in \(\ball(0,\delta)\). Define
\[
\pi_i^{\varepsilon,\theta}
:=
f_i^2
\xi_i^{\varepsilon,\theta}(x-x')
\mathbf 1_{(I_i^\delta\times I_i)\cup(I_i\times I_i^\delta)}(x,x')
\,dx\,dx',
\qquad
\pi_{\rm lead}^{\varepsilon,\theta}
:=
\sum_{i=1}^N\pi_i^{\varepsilon,\theta}.
\]
The measures \(\pi_i^{\varepsilon,\theta}\) are symmetric. Moreover, for
\(x\in I_i^\delta\),
\[
\frac{d(\operatorname{pr}_1)_\#\pi_i^{\varepsilon,\theta}}{d\cL^d}(x)
=
(1-\theta)f_i,
\]
whereas, for \(x\in I_i\),
\[
0\le
\frac{d(\operatorname{pr}_1)_\#\pi_i^{\varepsilon,\theta}}{d\cL^d}(x)
\le
(1-\theta)f_i
\le f_\mu(x).
\]
Since the cubes are disjoint up to null sets,
\[
(\operatorname{pr}_1)_\#\pi_{\rm lead}^{\varepsilon,\theta}
=
(\operatorname{pr}_2)_\#\pi_{\rm lead}^{\varepsilon,\theta}
\le \mu .
\]
Set
\[
\lambda_{\rm rem}^{\varepsilon,\theta}
:=
\mu-(\operatorname{pr}_1)_\#\pi_{\rm lead}^{\varepsilon,\theta}
=
\mu-(\operatorname{pr}_2)_\#\pi_{\rm lead}^{\varepsilon,\theta}.
\]

We next complete the remaining marginal mass. Let
\[
B:=\bigcup_{i=1}^N I_i^\delta,
\qquad
G:=\Omega_0\setminus B.
\]
Since \(|\partial\Omega_0|=0\), the sets \(B\) and \(G\) carry all of
\(\lambda_{\rm rem}^{\varepsilon,\theta}\). On \(I_i^\delta\), the density of
\(\lambda_{\rm rem}^{\varepsilon,\theta}|_B\) is independent of
\(\varepsilon\), for all sufficiently small \(\varepsilon\), and equals
\[
\rho_i^\theta(x):=f_\mu(x)-(1-\theta)f_i.
\]
By \eqref{eq:oscillation-large-p},
\[
\theta\underline f_\mu
\le
\theta f_i
\le
\rho_i^\theta(x)
\le
2\theta f_i
\le
2\theta\overline f_\mu,
\qquad x\in I_i^\delta .
\]
Applying \Cref{lemma:weighted-quantization-large-p} on each \(I_i^\delta\),
with reference measure \(f_\mu\,d\cL^d\), gives a coupling
\[
\pi_{\rm bulk}^{\varepsilon,\theta}
\in
\Pi(\lambda_{\rm rem}^{\varepsilon,\theta}|_B,
    \lambda_{\rm rem}^{\varepsilon,\theta}|_B)
\]
such that
\begin{equation}
    \label{eq:bulk-large-p-final}
\int\|x-x'\|^2\,d\pi_{\rm bulk}^{\varepsilon,\theta}
+
\varepsilon
\int
\left(
\frac{d\pi_{\rm bulk}^{\varepsilon,\theta}}{d(\mu\otimes\mu)}
\right)^p
d\mu d\mu
\le
C \theta^\kappa \varepsilon^{2/D},
\end{equation}
where \(C\) is independent of \(\theta,\mathcal P,\delta\), and
\(\varepsilon\).

Let \(P_G:=\mu|_G\) and
\(\lambda_G^{\varepsilon,\theta}:=\lambda_{\rm rem}^{\varepsilon,\theta}|_G\).
The density of \(P_G\) is \(f_\mu\), hence is bounded above and below on \(G\).
The density of \(\lambda_G^{\varepsilon,\theta}\) is bounded above by
\(\overline f_\mu\) and below by \(\theta\underline f_\mu\). Indeed, on
\(G\cap I_i\) the leading marginal density is at most \((1-\theta)f_i\), so
\[
\frac{d\lambda_G^{\varepsilon,\theta}}{d\cL^d}
\ge
f_\mu-(1-\theta)f_i
\ge
\theta f_i
\ge
\theta\underline f_\mu,
\]
whereas outside \(\bigcup_i I_i\) no leading mass is used. Therefore, for
fixed \(\theta\), \Cref{le:boundary-layer-large-p}, applied with
\(Q=\Omega_0\), \(P=P_G\), and
\(\lambda=\lambda_G^{\varepsilon,\theta}\), gives a coupling
\[
\pi_{\rm bdr}^{\varepsilon,\theta}
\in
\Pi(\lambda_{\rm rem}^{\varepsilon,\theta}|_G,
    \lambda_{\rm rem}^{\varepsilon,\theta}|_G)
\]
such that
\begin{equation}
    \label{eq:boundary-large-p-final}
\int\|x-x'\|^2\,d\pi_{\rm bdr}^{\varepsilon,\theta}
+
\varepsilon
\int
\left(
\frac{d\pi_{\rm bdr}^{\varepsilon,\theta}}{d(\mu\otimes\mu)}
\right)^p
d\mu d\mu
\le
C_\theta |G|\,\varepsilon^{2/D}.
\end{equation}
The constant \(C_\theta\) may depend on \(\theta\), but is independent of
\(\mathcal P,\delta\), and \(\varepsilon\).

Set
\[
\pi_{\rm rem}^{\varepsilon,\theta}
:=
\pi_{\rm bulk}^{\varepsilon,\theta}
+
\pi_{\rm bdr}^{\varepsilon,\theta},
\qquad
\pi_\varepsilon^\theta
:=
\pi_{\rm lead}^{\varepsilon,\theta}
+
\pi_{\rm rem}^{\varepsilon,\theta},
\]
and
\[
\gamma_\varepsilon^\theta
:=
(\operatorname{Id}\times\nabla\varphi)_\#\pi_\varepsilon^\theta .
\]
Then \(\pi_\varepsilon^\theta\in\Pi(\mu,\mu)\) and
\(\gamma_\varepsilon^\theta\in\Pi(\mu,\nu)\). By
\Cref{lemma:MA-change-of-variables}, applied to
\(\lambda=\pi_\varepsilon^\theta\), the \(L^p\)-penalty is invariant under
\((x,x')\mapsto(x,\nabla\varphi(x'))\). Hence
\[
\left\|
\frac{d\gamma_\varepsilon^\theta}{d(\mu\otimes\nu)}
\right\|_{L^p(\mu\otimes\nu)}^p
=
\left\|
\frac{d\pi_\varepsilon^\theta}{d(\mu\otimes\mu)}
\right\|_{L^p(\mu\otimes\mu)}^p .
\]
The leading densities are supported on the pairwise disjoint sets
\(I_i\times I_i\), up to null sets; hence their \(L^p\)-contributions add.
The bulk and boundary remainder densities are supported on the disjoint sets
\(B\times B\) and \(G\times G\). They need not, however, be disjoint from the
leading density. Applying
\((a+b)^p\le (1+\delta')a^p+C_{p,\delta'}b^p\), with \(a\) equal to the total
leading density and \(b\) equal to the total remainder density, bounds the
regularization term of \(\pi_\varepsilon^\theta\) by the leading contribution,
with factor \(1+\delta'\), plus a constant multiple of the two remainder
regularization terms. Also, by \Cref{lemma:integrated-bregman-identity},
\[
\int\|x-y\|^2\,d\gamma_\varepsilon^\theta(x,y)-{\rm OT}(\mu,\nu)
=
2\int\mathbb D(x,\nabla\varphi(x'))\,d\pi_\varepsilon^\theta(x,x'),
\]
and \Cref{lemma:Bregman_ptwisebd} gives
\(2\mathbb D(x,\nabla\varphi(x'))\le C\|x-x'\|^2\). Consequently,
\begin{equation}
    \label{eq:large-p-upper-split-final}
\frac{{\rm ROT}_{\varepsilon,p}(\mu,\nu)-{\rm OT}(\mu,\nu)}
{\varepsilon^{2/D}}
\le
\varepsilon^{-2/D}\mathcal T_{\varepsilon}^{\theta,\delta'}
+
C_{\delta'}\theta^\kappa
+
C_{\theta,\delta'}|G|,
\end{equation}
where
\[
\mathcal T_{\varepsilon}^{\theta,\delta'}
:=
\sum_{i=1}^N
\left[
2\int \mathbb D(x,\nabla\varphi(x'))\,d\pi_i^{\varepsilon,\theta}
+
(1+\delta')\varepsilon
\left\|
\frac{d\pi_i^{\varepsilon,\theta}}{d(\mu\otimes\mu)}
\right\|_{L^p(\mu\otimes\mu)}^p
\right].
\]

It remains to estimate \(\mathcal T_{\varepsilon}^{\theta,\delta'}\). Let
\(r_\varepsilon\) be the maximal radius of the supports of the profiles
\(\xi_i^{\varepsilon,\theta}\). For fixed \(\mathcal P,\delta,\theta\), we
have \(r_\varepsilon\to0\). By \Cref{lemma:Bregman_ptwisebd}, the uniform
ellipticity, and the uniform continuity of \(\nabla^2\varphi\), there are
numbers \(\eta_{\mathcal P}\to0\) as
\(\operatorname{mesh}(\mathcal P)\to0\) and \(\omega(r_\varepsilon)\to0\) as
\(\varepsilon\downarrow0\), such that, on the support of
\(\pi_i^{\varepsilon,\theta}\),
\[
2\mathbb D(x,\nabla\varphi(x'))
\le
\bigl(1+\eta_{\mathcal P}+\omega(r_\varepsilon)\bigr)
\|x-x'\|_{A_i}^2 .
\]
Moreover, since \(f_\mu\ge f_i\) on \(I_i\),
\[
\left\|
\frac{d\pi_i^{\varepsilon,\theta}}{d(\mu\otimes\mu)}
\right\|_{L^p(\mu\otimes\mu)}^p
\le
f_i^2
\int_{I_i}\int_{\mathbb R^d}
\bigl(\xi_i^{\varepsilon,\theta}(x-x')\bigr)^p\,dx'\,dx .
\]
Thus,
\[
\begin{aligned}
&2\int \mathbb D(x,\nabla\varphi(x'))\,d\pi_i^{\varepsilon,\theta}
+
(1+\delta')\varepsilon
\left\|
\frac{d\pi_i^{\varepsilon,\theta}}{d(\mu\otimes\mu)}
\right\|_{L^p(\mu\otimes\mu)}^p
\\
&\quad\le
\bigl(1+\delta'+\eta_{\mathcal P}+\omega(r_\varepsilon)\bigr)
|I_i|f_i^2
\frac1{\varepsilon^{1/(p-1)}}
\int_{\mathbb R^d}
\left[
\|z\|_{A_i}^2
\left(c_i^{\varepsilon,\theta}-\frac1p\|z\|_{A_i}^2\right)_+^{\frac1{p-1}}
\right.
\\[-.4em]
&\hspace{13em}\left.
+
\left(c_i^{\varepsilon,\theta}-\frac1p\|z\|_{A_i}^2\right)_+^{\frac p{p-1}}
\right]dz .
\end{aligned}
\]
The same beta-function computation as in \Cref{lemma:integrals} gives
\[
\begin{aligned}
&\frac1{\varepsilon^{1/(p-1)}}
\int_{\mathbb R^d}
\left[
\|z\|_{A_i}^2
\left(c_i^{\varepsilon,\theta}-\frac1p\|z\|_{A_i}^2\right)_+^{\frac1{p-1}}
+
\left(c_i^{\varepsilon,\theta}-\frac1p\|z\|_{A_i}^2\right)_+^{\frac p{p-1}}
\right]dz
\\
&\hspace{8em}
=
\mathfrak C_{d,p}
\bigl((1-\theta)f_i^{-1}\bigr)^\kappa
\varepsilon^{2/D}
(\det A_i)^\alpha .
\end{aligned}
\]
Therefore,
\begin{equation}
    \label{eq:leading-large-p-final}
\limsup_{\varepsilon\downarrow0}
\varepsilon^{-2/D}\mathcal T_{\varepsilon}^{\theta,\delta'}
\le
(1+\delta'+\eta_{\mathcal P})
(1-\theta)^\kappa
\mathfrak C_{d,p}
S_{\mathcal P},
\end{equation}
where
\[
S_{\mathcal P}
:=
\sum_{i=1}^N
|I_i|\,f_i^{1-2\alpha}(\det A_i)^\alpha .
\]

Since \(f_\mu\) is continuous and bounded away from zero on
\(\operatorname{supp}\mu\), and \(\nabla^2\varphi\) is continuous, the sums
\(S_{\mathcal P}\) converge, along any exhausting sequence of such cubical
families with \(\operatorname{mesh}(\mathcal P)\to0\), to the integral in
\eqref{eq:K-large-p-det-form}. Also,
\[
|G|
=
\left|\Omega_0\setminus\bigcup_{i=1}^N I_i^\delta\right|
\longrightarrow
\left|\Omega_0\setminus\bigcup_{i=1}^N I_i\right|
\qquad\text{as }\delta\downarrow0,
\]
and the latter tends to \(0\) along the same exhaustion of \(\Omega_0\).
Combining \eqref{eq:large-p-upper-split-final} and
\eqref{eq:leading-large-p-final}, then letting first \(\delta\downarrow0\) and
afterwards refining the cubical family, yields, for fixed
\(\theta\in(0,1)\) and \(\delta'>0\),
\[
\limsup_{\varepsilon\downarrow0}
\frac{{\rm ROT}_{\varepsilon,p}(\mu,\nu)-{\rm OT}(\mu,\nu)}
{\varepsilon^{2/D}}
\le
(1+\delta')
(1-\theta)^\kappa
\mathfrak C_{d,p}\mathcal K(\mu,\nu,p)
+
C_{\delta'}\theta^\kappa .
\]
Finally let \(\theta\downarrow0\) and then \(\delta'\downarrow0\). This gives
\[
\limsup_{\varepsilon\downarrow0}
\frac{{\rm ROT}_{\varepsilon,p}(\mu,\nu)-{\rm OT}(\mu,\nu)}
{\varepsilon^{2/(d(p-1)+2)}}
\le
\mathfrak C_{d,p}\mathcal K(\mu,\nu,p).
\]
Together with the lower bound in \Cref{thm:lower_bound}, this proves
\Cref{theorem:main-p} for \(p>2\).
\end{proof}

\subsection{Proof of \Cref{theorem:main-EOT} on EOT}

\begin{proof}[Proof of \Cref{theorem:main-EOT}]
It remains to prove the upper bound. Assume
\ref{assumption:hoelder_densities}, \ref{assumption:elliptic_phi}, and \(\mu,\nu\in\mathcal P_{2+\beta}(\mathbb R^d)\), and set \(T:=\nabla\varphi\). For each \(R\ge1\), choose finite nonnegative simple functions
\[
    f_{n,R}
    =
    \sum_{i\in\mathcal I_{n,R}}
    a_{n,i,R}\mathbf 1_{E_{n,i,R}},
\]
where \(\mathcal I_{n,R}\) is finite and 
\(\{E_{n,i,R}\}_{i\in\mathcal I_{n,R}}\) are disjoint Borel sets such that
\[
    0\le f_{n,R}\le f_{n+1,R}\le f_\mu,
    \qquad
    f_{n,R}\uparrow f_\mu\mathbf 1_{[-R,R]^d}
    \quad
    \text{\(\cL^d\)-a.e.\ and in }L^1(\mathbb R^d),
\]
and
\[
    \delta_{n,R}
    :=
    \sup_{i\in\mathcal I_{n,R}}
    \operatorname{diam}(E_{n,i,R})
    \to0,
    \qquad
    \bigcup_{i\in\mathcal I_{n,R}}E_{n,i,R}
    \subset[-R,R]^d .
\]
We discard indices with \(a_{n,i,R}=0\) or \(|E_{n,i,R}|=0\). By monotone
convergence and the finite \((2+\beta)\)-moment assumption,
\[
    \lim_{R\uparrow\infty}\lim_{n\uparrow\infty}
    \int \|x\|^{2+\beta}
    \bigl(f_\mu-f_{n,R}\bigr)(x)\,dx
    =0 .
\]

We shall also use the corresponding entropy convergence. Since
\(t\mapsto(t\log t)^+\) is nondecreasing on \([0,\infty)\), monotone
convergence gives
\[
    \int (f_{n,R}\log f_{n,R})^+\,dx
    \longrightarrow
    \int_{[-R,R]^d}(f_\mu\log f_\mu)^+\,dx .
\]
Letting \(R\uparrow\infty\) and using monotone convergence once more,
\[
    \lim_{R\uparrow\infty}\lim_{n\uparrow\infty}
    \int (f_{n,R}\log f_{n,R})^+\,dx
    =
    \int (f_\mu\log f_\mu)^+\,dx .
\]
For the negative part, choose \(s\in(0,1)\) such that
\(f_\mu^s\in L^1(\mathbb R^d)\), which is possible by the moment assumption as
in \Cref{rk:finitenessOfLimit}. Since
\[
    |t\log t|\mathbf 1_{\{0<t\le1\}}
    \le
    \frac{1}{e(1-s)}t^s
\]
and \(0\le f_{n,R}\le f_\mu\), we have
\[
    (f_{n,R}\log f_{n,R})^-
    \le
    \frac{1}{e(1-s)}f_\mu^s .
\]
Thus, for fixed \(R\), dominated convergence yields
\[
    \int (f_{n,R}\log f_{n,R})^-\,dx
    \longrightarrow
    \int_{[-R,R]^d}(f_\mu\log f_\mu)^-\,dx .
\]
Letting \(R\uparrow\infty\) and using dominated convergence again,
\[
    \lim_{R\uparrow\infty}\lim_{n\uparrow\infty}
    \int (f_{n,R}\log f_{n,R})^-\,dx
    =
    \int (f_\mu\log f_\mu)^-\,dx .
\]
Consequently,
\begin{equation}
    \label{eq:KL-conv-R}
    \lim_{R\uparrow\infty}\lim_{n\uparrow\infty}
    \int f_{n,R}\log f_{n,R}\,dx
    =
    \int f_\mu\log f_\mu\,dx
    \in(-\infty,\infty].
\end{equation}

Fix \(R<\infty\) and \(n\in\mathbb N\). To simplify notation, write
\[
    E_i:=E_{n,i,R},
    \qquad
    a_i:=a_{n,i,R},
    \qquad
    \mathcal I:=\mathcal I_{n,R}.
\]
For each \(i\in\mathcal I\), choose \(x_i\in E_i\) and set
\(A_i:=\nabla^2\varphi(x_i)\). Define
\[
    G_i^\varepsilon(z)
    :=
    G_{\varepsilon,a_i,A_i}(z)
    =
    \frac{\sqrt{\det A_i}}{a_i(\pi\varepsilon)^{d/2}}
    \exp\left(-\frac{\|z\|_{A_i}^2}{\varepsilon}\right).
\]
By \Cref{lemma:eot-gaussian-profile},
\[
    \int_{\mathbb R^d}G_i^\varepsilon(z)\,dz=\frac1{a_i}.
\]
Define the leading measure on \(E_i\times E_i\) by
\[
    \pi_i^\varepsilon(dx\,dx')
    :=
    a_i^2
    \mathbf 1_{E_i}(x)\mathbf 1_{E_i}(x')
    G_i^\varepsilon(x-x')\,dx\,dx',
\]
and set
\[
    \pi_{\rm lead}^\varepsilon
    :=
    \sum_{i\in\mathcal I}\pi_i^\varepsilon .
\]
For \(x\in E_i\),
\[
    \frac{d(\operatorname{pr}_1)_\#\pi_i^\varepsilon}{d\cL^d}(x)
    =
    a_i^2\mathbf 1_{E_i}(x)
    \int_{E_i}G_i^\varepsilon(x-x')\,dx'
    \le
    a_i\mathbf 1_{E_i}(x)
    \le
    f_\mu(x)
\]
for \(\cL^d\)-a.e.\ \(x\). The same estimate holds for the second marginal,
by symmetry. Hence
\[
    \lambda_{\rm rem}^\varepsilon
    :=
    \mu-(\operatorname{pr}_1)_\#\pi_{\rm lead}^\varepsilon
    =
    \mu-(\operatorname{pr}_2)_\#\pi_{\rm lead}^\varepsilon
\]
is a finite positive measure.

Choose \(L<\infty\) such that
\[
    2\mathbb D(x,T(x'))
    \le
    L\|x-x'\|^2,
    \qquad x,x'\in\mathbb R^d,
\]
which exists by \Cref{lemma:Bregman_ptwisebd}. If
\(\lambda_{\rm rem}^\varepsilon(\mathbb R^d)=0\), set
\(\pi_{\rm rem}^\varepsilon=0\). Otherwise, for all sufficiently small
\(\varepsilon\), choose
\[
    \pi_{\rm rem}^\varepsilon
    \in
    \Pi(\lambda_{\rm rem}^\varepsilon,\lambda_{\rm rem}^\varepsilon)
\]
according to the construction underlying \Cref{cor:eot-quant-scaled}, with
this value of \(L\) and regularization parameter \(\varepsilon\).

Set
\[
    \pi_{n,R}^\varepsilon
    :=
    \pi_{\rm lead}^\varepsilon+\pi_{\rm rem}^\varepsilon,
    \qquad
    \gamma_{n,R}^\varepsilon
    :=
    (\operatorname{Id}\times T)_\#\pi_{n,R}^\varepsilon .
\]
Then \(\pi_{n,R}^\varepsilon\in\Pi(\mu,\mu)\) and
\(\gamma_{n,R}^\varepsilon\in\Pi(\mu,\nu)\).

We first estimate the leading contribution. Since \(f_{n,R}\le f_\mu\), on
\(E_i\times E_i\) we have
\[
    \frac{d\pi_i^\varepsilon}{d(\mu\otimes\mu)}(x,x')
    =
    \frac{a_i^2G_i^\varepsilon(x-x')}
         {f_\mu(x)f_\mu(x')}
    \le
    G_i^\varepsilon(x-x')
\]
for \((\mu\otimes\mu)\)-a.e.\ \((x,x')\). Therefore
\[
    {\rm KL}(\pi_i^\varepsilon\mid\mu\otimes\mu)
    \le
    \int_{E_i\times E_i}
    \log G_i^\varepsilon(x-x')\,
    d\pi_i^\varepsilon(x,x').
\]

The local Taylor estimate for the Bregman divergence is used on the
compact \([-R,R]^d\). Let \(\omega_R\) be a nondecreasing modulus such
that
\[
    \left|
    2\mathbb D(x,T(x'))
    -
    \|x-x'\|_{\nabla^2\varphi(x)}^2
    \right|
    \le
    \omega_R(\|x-x'\|)\|x-x'\|^2,
    \qquad x,x'\in[-R,R]^d .
\]
Define
\[
    \omega_{n,R}
    :=
    \omega_R(\delta_{n,R})
    +
    \sup_{\substack{u,v\in[-R,R]^d\\ \|u-v\|\le \delta_{n,R}}}
    \|\nabla^2\varphi(u)-\nabla^2\varphi(v)\| .
\]
Since \(\nabla^2\varphi\) is continuous on \([-R,R]^d\),
\[
    \omega_{n,R}\to0
    \qquad\text{as }n\uparrow\infty
    \quad\text{for every fixed }R.
\]
For \(x,x'\in E_i\), this gives
\[
    \left|
    2\mathbb D(x,T(x'))
    -
    \|x-x'\|_{A_i}^2
    \right|
    \le
    \omega_{n,R}\|x-x'\|^2 .
\]
Using the uniform ellipticity of \(A_i\) and the Gaussian second-moment bound,
we obtain
\begin{equation}
    \label{eq:bound-EOT-InBR-R}
    \sum_{i\in\mathcal I}
    \int_{E_i\times E_i}
    \left|
    2\mathbb D(x,T(x'))
    -
    \|x-x'\|_{A_i}^2
    \right|
    d\pi_i^\varepsilon(x,x')
    \le
    C\omega_{n,R}\varepsilon ,
\end{equation}
where \(C\) is independent of \(\varepsilon,n\), and \(R\). Consequently,
\begin{multline*}
    \sum_{i\in\mathcal I}
    \left[
    2\int \mathbb D(x,T(x'))\,d\pi_i^\varepsilon
    +
    \varepsilon\,{\rm KL}(\pi_i^\varepsilon\mid\mu\otimes\mu)
    \right]
    \\
    \le
    \sum_{i\in\mathcal I}
    \int_{E_i\times E_i}
    \left[
    \|x-x'\|_{A_i}^2
    +
    \varepsilon\log G_i^\varepsilon(x-x')
    \right]
    d\pi_i^\varepsilon(x,x')
    +
    C\omega_{n,R}\varepsilon .
\end{multline*}
Since
\[
    \log G_i^\varepsilon(z)
    =
    \frac12\log\det A_i
    -
    \log a_i
    -
    \frac d2\log(\pi\varepsilon)
    -
    \frac{\|z\|_{A_i}^2}{\varepsilon},
\]
we have the pointwise cancellation
\[
    \|z\|_{A_i}^2+\varepsilon\log G_i^\varepsilon(z)
    =
    \varepsilon
    \left[
    \frac12\log\det A_i
    -
    \log a_i
    -
    \frac d2\log(\pi\varepsilon)
    \right].
\]
Writing
\[
    m_i^\varepsilon
    :=
    \pi_i^\varepsilon(\mathbb R^d\times\mathbb R^d),
\]
we therefore obtain
\begin{multline*}
    \sum_{i\in\mathcal I}
    \left[
    2\int \mathbb D(x,T(x'))\,d\pi_i^\varepsilon
    +
    \varepsilon\,{\rm KL}(\pi_i^\varepsilon\mid\mu\otimes\mu)
    \right]
    \\
    \le
    \varepsilon
    \sum_{i\in\mathcal I}
    m_i^\varepsilon
    \left[
    \frac12\log\det A_i
    -
    \log a_i
    -
    \frac d2\log(\pi\varepsilon)
    \right]
    +
    C\omega_{n,R}\varepsilon .
\end{multline*}

We next record the convergence of the leading masses. Define
\[
    \eta_i^\varepsilon(z):=a_iG_i^\varepsilon(z).
\]
Then \((\eta_i^\varepsilon)_{\varepsilon>0}\) is an approximate identity.
The Lebesgue differentiation theorem applied to \(\mathbf 1_{E_i}\) yields
\[
    \eta_i^\varepsilon\ast\mathbf 1_{E_i}(x)
    \longrightarrow
    \mathbf 1_{E_i}(x)
    \qquad
    \text{for \(\cL^d\)-a.e.\ }x .
\]
Since \(0\le\eta_i^\varepsilon\ast\mathbf 1_{E_i}\le1\), dominated
convergence gives
\[
    m_i^\varepsilon
    =
    a_i\int_{E_i}
    (\eta_i^\varepsilon\ast\mathbf 1_{E_i})(x)\,dx
    \longrightarrow
    a_i|E_i|.
\]
Moreover,
\[
    (\operatorname{pr}_1)_\#\pi_{\rm lead}^\varepsilon
    \longrightarrow
    f_{n,R}\,d\cL^d
\]
in total variation. Since the first marginal of the leading part is supported
on \([-R,R]^d\), the same convergence holds against the weight
\(\|x\|^{2+\beta}\).

We now control the remainder. Let
\[
    m_{\rm rem}^\varepsilon
    :=
    \lambda_{\rm rem}^\varepsilon(\mathbb R^d),
    \qquad
    M_{\rm rem}^\varepsilon
    :=
    \int \|x\|^{2+\beta}\,
    d\lambda_{\rm rem}^\varepsilon(x).
\]
Since \(\lambda_{\rm rem}^\varepsilon\le\mu\), we have
\[
    {\rm KL}(\pi_{\rm rem}^\varepsilon\mid\mu\otimes\mu)
    \le
    {\rm KL}
    \bigl(
    \pi_{\rm rem}^\varepsilon
    \mid
    \lambda_{\rm rem}^\varepsilon
    \otimes
    \lambda_{\rm rem}^\varepsilon
    \bigr).
\]
Combining this with \(2\mathbb D(x,T(x'))\le L\|x-x'\|^2\) and
\Cref{cor:eot-quant-scaled}, and absorbing the difference between
\(\log\varepsilon\) and \(\log(\pi\varepsilon)\) into the linear term, we get
\begin{multline*}
    2\int \mathbb D(x,T(x'))\,
    d\pi_{\rm rem}^\varepsilon
    +
    \varepsilon\,
    {\rm KL}(\pi_{\rm rem}^\varepsilon\mid\mu\otimes\mu)
    \\
    \le
    -\frac d2\,m_{\rm rem}^\varepsilon
    \varepsilon\log(\pi\varepsilon)
    +
    C\varepsilon
    \left[
    m_{\rm rem}^\varepsilon
    +
    M_{\rm rem}^\varepsilon
    +
    m_{\rm rem}^\varepsilon|\log m_{\rm rem}^\varepsilon|
    \right].
\end{multline*}

It remains to combine the leading and remainder entropies. We use the
log-sum inequality in the following form: if \(\alpha,\beta\ll P\) are finite
positive measures with masses \(m_\alpha,m_\beta\) and
\(m_\alpha+m_\beta=1\), then
\[
    {\rm KL}(\alpha+\beta\mid P)
    \le
    {\rm KL}(\alpha\mid P)
    +
    {\rm KL}(\beta\mid P)
    +
    \mathfrak h(m_\beta),
\]
where
\[
    \mathfrak h(r)
    :=
    -r\log r-(1-r)\log(1-r),
    \qquad r\in[0,1].
\]
The leading pieces are supported on the pairwise disjoint sets
\(E_i\times E_i\), up to null sets, so
\[
    {\rm KL}(\pi_{\rm lead}^\varepsilon\mid\mu\otimes\mu)
    =
    \sum_{i\in\mathcal I}
    {\rm KL}(\pi_i^\varepsilon\mid\mu\otimes\mu).
\]
Applying the log-sum inequality with $\alpha=\pi_{\rm lead}^\varepsilon$ and $\beta=\pi_{\rm rem}^\varepsilon$ and $P=\mu\otimes\mu$ yields
\[
    {\rm KL}(\pi_{n,R}^\varepsilon\mid\mu\otimes\mu)
    \le
    {\rm KL}(\pi_{\rm lead}^\varepsilon\mid\mu\otimes\mu)
    +
    {\rm KL}(\pi_{\rm rem}^\varepsilon\mid\mu\otimes\mu)
    +
    \mathfrak h(m_{\rm rem}^\varepsilon).
\]

By \Cref{lemma:MA-change-of-variables}, applied to
\(\lambda=\pi_{n,R}^\varepsilon\), relative entropy is invariant under
\((x,x')\mapsto(x,T(x'))\). Hence
\[
    {\rm KL}
    \bigl(
    \gamma_{n,R}^\varepsilon
    \mid
    \mu\otimes\nu
    \bigr)
    =
    {\rm KL}
    \bigl(
    \pi_{n,R}^\varepsilon
    \mid
    \mu\otimes\mu
    \bigr).
\]
Furthermore, by \Cref{lemma:integrated-bregman-identity},
\[
    \int \|x-y\|^2\,
    d\gamma_{n,R}^\varepsilon(x,y)
    -
    {\rm OT}(\mu,\nu)
    =
    2\int \mathbb D(x,T(x'))\,
    d\pi_{n,R}^\varepsilon(x,x').
\]
Combining the preceding estimates and using
\[
    \sum_{i\in\mathcal I}m_i^\varepsilon
    +
    m_{\rm rem}^\varepsilon
    =
    1,
\]
we obtain
\begin{multline*}
    \frac{
    {\rm EOT}_\varepsilon(\mu,\nu)
    -
    {\rm OT}(\mu,\nu)
    +
    \frac d2\,\varepsilon\log(\pi\varepsilon)
    }{\varepsilon}
    \\
    \le
    \sum_{i\in\mathcal I}
    m_i^\varepsilon
    \left[
    \frac12\log\det A_i-\log a_i
    \right]
    +
    C\omega_{n,R}
    \\
    +
    C
    \left[
    m_{\rm rem}^\varepsilon
    +
    M_{\rm rem}^\varepsilon
    +
    m_{\rm rem}^\varepsilon|\log m_{\rm rem}^\varepsilon|
    \right]
    +
    \mathfrak h(m_{\rm rem}^\varepsilon).
\end{multline*}
Letting \(\varepsilon\downarrow0\) with \(n\) and \(R\) fixed gives
\[
    m_{\rm rem}^\varepsilon
    \longrightarrow
    r_{n,R}
    :=
    \int
    \bigl(f_\mu-f_{n,R}\bigr)(x)\,dx,
\]
\[
    M_{\rm rem}^\varepsilon
    \longrightarrow
    s_{n,R}
    :=
    \int
    \|x\|^{2+\beta}
    \bigl(f_\mu-f_{n,R}\bigr)(x)\,dx .
\]
Therefore
\begin{multline}
    \label{eq:EOT-limsup-fixed-nR}
    \limsup_{\varepsilon\downarrow0}
    \frac{
    {\rm EOT}_\varepsilon(\mu,\nu)
    -
    {\rm OT}(\mu,\nu)
    +
    \frac d2\,\varepsilon\log(\pi\varepsilon)
    }{\varepsilon}
    \\
    \le
    \sum_{i\in\mathcal I}
    a_i|E_i|
    \left[
    \frac12\log\det A_i-\log a_i
    \right]
    +
    C\omega_{n,R}
    \\
    +
    C
    \left[
    r_{n,R}
    +
    s_{n,R}
    +
    r_{n,R}|\log r_{n,R}|
    \right]
    +
    \mathfrak h(r_{n,R}).
\end{multline}

We now let \(n\uparrow\infty\) with \(R\) fixed. Since
\(f_{n,R}\uparrow f_\mu\mathbf 1_{[-R,R]^d}\) in \(L^1\),
\[
    r_{n,R}
    \longrightarrow
    r_R
    :=
    \int_{([-R,R]^d)^c}f_\mu(x)\,dx,
\]
and, by dominated convergence,
\[
    s_{n,R}
    \longrightarrow
    s_R
    :=
    \int_{([-R,R]^d)^c}
    \|x\|^{2+\beta}f_\mu(x)\,dx .
\]
Moreover, \(\omega_{n,R}\to0\). By uniform ellipticity,
\(\log\det\nabla^2\varphi\) is bounded on \(\mathbb R^d\), and by continuity it
is uniformly continuous on \([-R,R]^d\). Since \(\delta_{n,R}\to0\),
\[
    \sum_{i\in\mathcal I}
    a_i|E_i|\,
    \frac12\log\det\nabla^2\varphi(x_i)
    \longrightarrow
    \int_{[-R,R]^d}
    \frac12\log\det\nabla^2\varphi(x)f_\mu(x)\,dx .
\]
Also,
\[
    \sum_{i\in\mathcal I}a_i|E_i|\log a_i
    =
    \int f_{n,R}\log f_{n,R}\,dx
    \longrightarrow
    \int_{[-R,R]^d}f_\mu\log f_\mu\,dx
\]
in the extended sense described above. Letting \(n\uparrow\infty\) in
\eqref{eq:EOT-limsup-fixed-nR}, we obtain
\begin{multline*}
    \limsup_{\varepsilon\downarrow0}
    \frac{
    {\rm EOT}_\varepsilon(\mu,\nu)
    -
    {\rm OT}(\mu,\nu)
    +
    \frac d2\,\varepsilon\log(\pi\varepsilon)
    }{\varepsilon}
    \\
    \le
    \int_{[-R,R]^d}
    \left[
    \frac12\log\det\nabla^2\varphi(x)
    -
    \log f_\mu(x)
    \right]d\mu(x)
    \\
    +
    C
    \left[
    r_R+s_R+r_R|\log r_R|
    \right]
    +
    \mathfrak h(r_R).
\end{multline*}

Finally, let \(R\uparrow\infty\). Since
\(\mu\in\mathcal P_{2+\beta}(\mathbb R^d)\), we have
\(r_R\to0\) and \(s_R\to0\), hence \(\mathfrak h(r_R)\to0\) and \(r_R|\log r_R|\to0\). Since
\(\log\det\nabla^2\varphi\) is bounded by uniform ellipticity,
\[
    \int_{[-R,R]^d}
    \frac12\log\det\nabla^2\varphi(x)\,d\mu(x)
    \longrightarrow
    \int
    \frac12\log\det\nabla^2\varphi(x)\,d\mu(x)\in (-\infty,\infty).
\]
Together with the entropy convergence \eqref{eq:KL-conv-R}, this gives
\[
    \limsup_{\varepsilon\downarrow0}
    \frac{
    {\rm EOT}_\varepsilon(\mu,\nu)
    -
    {\rm OT}(\mu,\nu)
    +
    \frac d2\,\varepsilon\log(\pi\varepsilon)
    }{\varepsilon}
    \le
    \int
    \left[
    \frac12\log\det\nabla^2\varphi(x)
    -
    \log f_\mu(x)
    \right]d\mu(x),
\]
where the right-hand side is in $[-\infty,\infty)$. By \eqref{eq:MA-density-identity},
\[
    \frac12\log\det\nabla^2\varphi(x)-\log f_\mu(x)
    =
    -\frac12
    \log\bigl[
    f_\mu(x)f_\nu(T(x))
    \bigr],
\]
and so
\[
    \limsup_{\varepsilon\downarrow0}
    \frac{
    {\rm EOT}_\varepsilon(\mu,\nu)
    -
    {\rm OT}(\mu,\nu)
    +
    \frac d2\,\varepsilon\log(\pi\varepsilon)
    }{\varepsilon}
    \le
    \mathcal K_{\rm EOT}(\mu,\nu).
\]
Together with the lower bound in \Cref{thm:eot-lower-bound}, this proves
\Cref{theorem:main-EOT}.
\end{proof}

\section{Omitted proofs}\label{se:omitted-proofs}

\begin{proof}[Proof of \cref{rk:finitenessOfLimit}]
We first show the sufficient conditions for finiteness. Recall that ${\rm Ent}(\alpha)>-\infty$ for any $\alpha\in\mathcal P^{\rm ac}_2(\mathbb R^d)$. If $f_\mu,f_\nu$ are bounded from above on the respective supports, ${\rm Ent}(\mu)$ and ${\rm Ent}(\nu)$ are finite, and then so is $\mathcal K_{\rm EOT}(\mu,\nu)$. For $\mathcal{K}(\mu,\nu,p)$, a lower bound leads to a bounded integrand and hence $\mathcal{K}(\mu,\nu,p)<\infty$.

For $p>1$, set
\[
\alpha_p:=\frac{p-1}{d(p-1)+2},\qquad s_p:=1-2\alpha_p=\frac{(d-2)(p-1)+2}{d(p-1)+2}.
\]
For a density $f$ and $s\in\mathbb R$, write $I_s(f):=\int_{\{f>0\}} f^s\,d\cL^d$; in particular, $I_0(f)=\cL^d(\{f>0\})$. Since $(\nabla\varphi)_\#\mu=\nu$, the inequality $ab\leq(a^2+b^2)/2$ gives
\[
\mathcal K(\mu,\nu,p)\leq \frac12 I_{s_p}(f_\mu)+\frac12 I_{s_p}(f_\nu).
\]
We shall use the following elementary consequence of H\"older's inequality: if $f$ is a probability density with finite $q$-moment and $s\in(d/(d+q),1)$, then $I_s(f)<\infty$. Indeed,
\[
\int f^s\,d\cL^d\leq \left(\int (1+\|x\|^q)f(x)\,d\cL^d(x)\right)^s
\left(\int (1+\|x\|^q)^{-s/(1-s)}\,d\cL^d(x)\right)^{1-s},
\]
and the last integral is finite whenever $qs/(1-s)>d$. If $p\in(1,2]$, then $s_p\geq s_2=d/(d+2)>d/(d+2+\beta)$, so the preceding estimate with $q=2+\beta$ proves $\mathcal K(\mu,\nu,p)<\infty$.

If $f\,d\cL^d$ is compactly supported and $s\in[0,1]$, then $I_s(f)\leq \cL^d(\operatorname{supp}(f\,d\cL^d))+1<\infty$. Since $s_p\in[0,1]$ for all $p>2$ when $d\geq2$, and $s_p=(3-p)/(p+1)$ when $d=1$, the same bound proves the compact-support assertion for $d\geq2$ and, in dimension one, for $p\leq3$.

It remains to construct the counterexamples. Let
\[
h(x):=\frac{\mathbf 1_{(0,e^{-1})}(x)}{x(\log(1/x))^2},
\]
and let $\eta$ be the product of $h(x)\,dx$ with the uniform law on $[0,1]^{d-1}$, with the latter factor omitted if $d=1$. The substitution $t=\log(1/x)$ shows that $\int h\,dx=1$ and
\[
{\rm Ent}(\eta)=\int_1^\infty \frac{t-2\log t}{t^2}\,dt=+\infty.
\]
Taking $\mu=\nu=\eta$, the Brenier map is the identity, so \cref{assumption:elliptic_phi} holds with $\varphi(x)=\|x\|^2/2$; moreover $\eta$ is compactly supported and satisfies \cref{assumption:hoelder_densities}.

For the $p=2$ counterexample, let $s_2=d/(d+2)$, choose $\gamma\in(1,(d+2)/d]$, and define
\[
f_2(x):=c_2(1+\|x\|)^{-(d+2)}\bigl(\log(e+\|x\|)\bigr)^{-\gamma}.
\]
Then $f_2\,d\cL^d\in\mathcal P_2^{\rm ac}(\mathbb R^d)$, since the radial second-moment integral is comparable at infinity to $\int^\infty r^{-1}(\log r)^{-\gamma}\,dr<\infty$. On the other hand,
\[
I_{s_2}(f_2)=\infty,
\]
because the corresponding radial integral is comparable at infinity to $\int^\infty r^{-1}(\log r)^{-\gamma s_2}\,dr$ and $\gamma s_2\leq1$. With $\mu=\nu=f_2\,d\cL^d$, the Brenier map is $\operatorname{Id}$ and $\mathcal K(\mu,\mu,2)=I_{s_2}(f_2)=\infty$.

Now fix $p>2$. Then $s_p<d/(d+2)$. If $s_p>0$, choose $\beta>0$ and $k$ such that $d+2+\beta<k\leq d/s_p$, which is possible since $d/s_p>d+2$; if $s_p\leq0$, choose any $\beta>0$ and $k>d+2+\beta$. Set
\[
f_p(x):=c_p(1+\|x\|)^{-k}.
\]
Then $f_p\,d\cL^d\in\mathcal P_{2+\beta}^{\rm ac}(\mathbb R^d)$. Moreover $I_{s_p}(f_p)=\infty$: when $s_p\neq0$, the radial integrand is comparable at infinity to $r^{d-1-ks_p}$, which is nonintegrable under the above choices, while for $s_p=0$ one has $I_0(f_p)=\cL^d(\mathbb R^d)$. Taking $\mu=\nu=f_p\,d\cL^d$ gives $\mathcal K(\mu,\mu,p)=\infty$, again with \cref{assumption:hoelder_densities,assumption:elliptic_phi} satisfied.

Finally, suppose $d=1$ and $p>3$, so $s_p=(3-p)/(p+1)<0$. Choose $a>0$ with $a s_p\leq -1$ and set
\[
f(x):=(a+1)x^a\mathbf 1_{(0,1)}(x).
\]
For $\mu=\nu=f\,d\cL^1$, the measures are compactly supported and satisfy \cref{assumption:hoelder_densities,assumption:elliptic_phi}, but
\[
\mathcal K(\mu,\mu,p)=I_{s_p}(f)=(a+1)^{s_p}\int_0^1 x^{a s_p}\,dx=\infty. \qedhere
\]
\end{proof}

\begin{proof}[Proof of \cref{le:consistency-p-to-one}]
Set
\[
\alpha:=p-1,
\qquad
\lambda_\alpha:=\frac{\alpha}{d\alpha+2},
\qquad
T:=\nabla\varphi,
\]
and write
\[
G(x):=f_\mu(x)f_\nu(T(x)),
\qquad
Z(x):=\log G(x),
\qquad
Z_\pm:=\max\{\pm Z,0\}.
\]
Then
\[
\mathcal K(\mu,\nu,1+\alpha)
=
\int e^{-\lambda_\alpha Z(x)}\,d\mu(x).
\]
Since \(T_\#\mu=\nu\), the logarithmic integrability assumption gives
\[
\int |Z|\,d\mu
\le
\int |\log f_\mu|\,d\mu+
\int |\log f_\nu|\,d\nu
<\infty,
\]
and therefore
\[
\mathcal K_{\rm EOT}(\mu,\nu)=-\frac12\int Z\,d\mu .
\]
Moreover, \cref{rk:finitenessOfLimit}(ii), evaluated at \(p=2\), yields
\(\int G^{-1/(d+2)}\,d\mu<\infty\).  Hence, with
\(\eta:=(d+2)^{-1}\),
\[
\int e^{\eta Z_-}\,d\mu<\infty,
\]
since \(e^{\eta Z_-}\le 1+G^{-\eta}\).  As
\(\lambda_\alpha/\alpha\to1/2\),
\[
\frac{e^{-\lambda_\alpha Z}-1}{\alpha}
\longrightarrow
-\frac12 Z
\qquad \mu\text{-a.e.}
\]
This convergence also holds in \(L^1(\mu)\).  Indeed, for all sufficiently small
\(\alpha\), on \(\{Z\ge0\}\),
\[
\left|\frac{e^{-\lambda_\alpha Z}-1}{\alpha}\right|
\le
\frac{\lambda_\alpha}{\alpha}Z
\le C Z_+,
\]
while on \(\{Z<0\}\), using \(e^u-1\le u e^u\) and
\(\lambda_\alpha\le\eta/2\),
\[
0\le
\frac{e^{\lambda_\alpha Z_-}-1}{\alpha}
\le
\frac{\lambda_\alpha}{\alpha}Z_-e^{\lambda_\alpha Z_-}
\le
C Z_-e^{\eta Z_-/2}
\le
C_\eta e^{\eta Z_-}.
\]
Dominated convergence then gives
\[
\mathcal K(\mu,\nu,1+\alpha)
=
1+
\alpha\mathcal K_{\rm EOT}(\mu,\nu)
+o(\alpha),
\]
and hence
\[
\log\mathcal K(\mu,\nu,1+\alpha)
=
\alpha\mathcal K_{\rm EOT}(\mu,\nu)+o(\alpha).
\]

It remains to expand the dimensional factor.  Since
\(B(x,b)=\Gamma(b)x^{-b}(1+O(x^{-1}))\) as \(x\to\infty\),
\[
\begin{aligned}
\mathfrak B_{d,1+\alpha}
&=
\frac{\omega_{d-1}}2(1+\alpha)^{d/2}
B\left(\frac{1+\alpha}{\alpha},\frac d2\right)  \\
&=
\frac{\omega_{d-1}\Gamma(d/2)}2\,\alpha^{d/2}(1+O(\alpha))
=(\pi\alpha)^{d/2}(1+O(\alpha)).
\end{aligned}
\]
Also
\[
R_\alpha:=
\frac{(1+\alpha)(d\alpha+2)}{d\alpha+2+2\alpha}
=1+O(\alpha^2).
\]
Thus, writing
\[
q_\alpha:=\frac{2}{d\alpha+2},
\]
and using \(\mathfrak C_{d,1+\alpha}=R_\alpha
\mathfrak B_{d,1+\alpha}^{-\alpha q_\alpha}\), we obtain
\[
\log\mathfrak C_{d,1+\alpha}
=
-\frac d2\alpha\log(\pi\alpha)
+O(\alpha^2|\log\alpha|).
\]
Define
\[
H_{\varepsilon,\alpha}
:=
\mathfrak C_{d,1+\alpha}
\varepsilon^{q_\alpha-1}
\alpha^{1-q_\alpha}
\mathcal K(\mu,\nu,1+\alpha).
\]
Since \(1-q_\alpha=d\alpha/(d\alpha+2)=(d/2)\alpha+O(\alpha^2)\), for fixed
\(\varepsilon>0\),
\[
\begin{aligned}
\log H_{\varepsilon,\alpha}
&=
\log\mathfrak C_{d,1+\alpha}
+(q_\alpha-1)\log\varepsilon
+(1-q_\alpha)\log\alpha
+\log\mathcal K(\mu,\nu,1+\alpha) \\
&=
\alpha\left[
\mathcal K_{\rm EOT}(\mu,\nu)
-
\frac d2\log(\pi\varepsilon)
\right]
+o(\alpha).
\end{aligned}
\]
Consequently,
\[
H_{\varepsilon,\alpha}
=
1+
\alpha\left[
\mathcal K_{\rm EOT}(\mu,\nu)
-
\frac d2\log(\pi\varepsilon)
\right]
+o(\alpha).
\]
Finally,
\[
\mathfrak C_{d,1+\alpha}
\left(\frac{\varepsilon}{\alpha}\right)^{q_\alpha}
\mathcal K(\mu,\nu,1+\alpha)
-
\frac{\varepsilon}{\alpha}
=
\frac{\varepsilon}{\alpha}\bigl(H_{\varepsilon,\alpha}-1\bigr),
\]
and the desired limit follows.
\end{proof}

\bibliographystyle{abbrv}
\bibliography{biblio1}

\end{document}